\documentclass[preprint,12pt]{elsarticle}
\usepackage{amsmath}
\usepackage{a4wide}
\usepackage[utf8]{inputenc}
\usepackage{amssymb}
\usepackage{amsopn}
\usepackage{epsfig}
\usepackage{amsfonts}
\usepackage{float}
\usepackage{latexsym}
\usepackage{amsthm}
\usepackage{enumerate}
\usepackage[UKenglish]{babel}
\usepackage{verbatim}
\usepackage{color}
\usepackage{calrsfs}
\usepackage{pgf}
\usepackage{tikz}
\usepackage{epstopdf}
\usepackage{graphicx}

\usepackage{mathrsfs}
 \usetikzlibrary{arrows,automata}
 \DeclareMathAlphabet{\pazocal}{OMS}{zplm}{m}{n}

\newtheorem{theorem}{Theorem}[section]
\newtheorem{lemma}[theorem]{Lemma}
\newtheorem{proposition}[theorem]{Proposition}

 \newtheorem{main}{Theorem}

\theoremstyle{definition}
\newtheorem{definition}[theorem]{Definition}

\newtheorem{question}{Question}

\theoremstyle{remark}
\newtheorem{remark}[theorem]{Remark}

\numberwithin{equation}{section}
\newcommand{\R}{\ensuremath{\mathbb{R}}}
\newcommand{\N}{\ensuremath{\mathbb{N}}}

\newcommand{\bi}{\mathbf{i}}
\newcommand{\bj}{ {\mathbf{j}}}
\newcommand{\bk}{ {\mathbf{k}}}
\newcommand{\sla}{ {\tilde\lambda}}

\renewcommand{\u}{\ensuremath{\pazocal{U}}}

\newcommand{\set}[1]{\left\{#1\right\}}
\newcommand{\la}{\lambda}

\newcommand{\Ga}{\Gamma}
\newcommand{\ep}{\varepsilon}
\newcommand{\f}{\infty}

\newcommand{\Om}{\Omega_\lambda}
\newcommand{\De}{I}

\newcommand{\La}{\Lambda}
\newcommand{\ra}{\rightarrow}
\usepackage{amssymb}

\begin{document}

\begin{frontmatter}



\title{On the structure of $\lambda$-Cantor set with overlaps}


\author[UU]{Karma~Dajani}
\ead{k.dajani1@uu.nl}
\address[UU]{Department of Mathematics, Utrecht University, Fac Wiskunde en informatica and MRI, Budapestlaan 6, P.O. Box 80.000, 3508 TA Utrecht, The Netherlands}

\author[CU]{Derong~Kong}
\ead{derongkong@126.com}
\address[CU]{College of Mathematics and Statistics, Chongqing University, 401331, Chongqing, P.R.China}

\author[ECUST,UU]{Yuanyuan~Yao\corref{cor1}}
\ead{yaoyuanyuan@ecust.edu.cn} \cortext[cor1]{Corresponding author}
\address[ECUST]{Department of Mathematics, East China University of Science and Technology, Shanghai 200237, P.R. China}

\address{}

\begin{abstract}
Given $\la\in(0, 1)$, let $E_\la$ be the self-similar set generated by the iterated function system  $\{x/3,(x+\la)/3,(x+2)/3\}$. Then $E_\la$ is a self-similar set with overlaps. We obtain the sufficient and necessary {condition} for $E_\la$ to be {\it totally self-similar}, which is a concept first introduced by Broomhead, Montaldi, and Sidorov in 2004. When $E_\la$ is totally self-similar, all its generating IFSs are investigated, and the size of the set of points having finite triadic codings is determined. Besides, we give some properties of the spectrum of $E_\la$ and show that the spectrum of $E_\la$ vanishes if and only if $\la$ is irrational.

\end{abstract}

\begin{keyword}
Self-similar sets with overlaps\sep totally self-similar\sep lower spectrum\sep generating iterated function systems\sep finite codings



\end{keyword}

\end{frontmatter}


 \section{Introduction}\label{s:introduction}

This paper studies the properties of a kind of overlapping self-similar set. Given $\la\in(0, 1)$, the \emph{$\la$-Cantor set} $E_\la$ is  the  {self-similar set} generated by the iterated function system
 \begin{equation}\label{eq:fd-similitude}
 f_d(x):=\frac{x+d}{3},\quad d\in\Om:=\set{0,\la, 2}.
 \end{equation}
 Then $E_\la$ is the unique non-empty compact set in $\R$ satisfying $E_\la=\bigcup_{d\in\Om}f_d(E_\la)$ (cf.~\cite{Hut-81}). Since $\la\in(0, 1)$, one can see that $f_0(I)\cap f_\la(I)\ne\emptyset$, where $I:=[0, 1]$ is the convex hull of $E_\la$. So $E_\la$ is a self-similar set with overlaps.

The interest in $E_\la$ stems from a conjecture of H.~Furstenberg. In the 1970s, he  conjectured that $\dim_HE_\la=1$ for all irrational $\la$  (see e.g. \cite{Per-Sol-00}, Question 2.5); this was partially answered by Kenyon \cite{Ken-97} and was finally proved by B. Solomyak and P. Shmerkin (their proof was included in the work of Hochman \cite{Hoc-14}).

The majority of the work done concerns the Hausdorff dimension and measure of $E_\la$ (see \cite{Ken-97,Rao-Wen-98,Swi-Vee-02}). However, in this paper we are interested in the following problems: what can we say about its  spectrum and all its generating iterated function systems? To make these questions clear, we review some standard terminology.

It follows from (\ref{eq:fd-similitude}) that for any $x\in E_\la$ there exists an infinite sequence $(d_i)$ over the set $\Omega_\la$  such that
 \[
 x=\lim_{n\ra\f}f_{d_1\ldots d_n}(0)=\sum_{i=1}^\f \frac{d_i}{3^i}=:((d_i))_3
 \]
 where $f_{d_1\ldots d_n}:=f_{d_1}\circ\cdots\circ f_{d_n}$ denotes the composition of $f_{d_1}, \ldots, f_{d_n}$. The infinite sequence $(d_i)$ is called a \emph{coding} of $x$ with respect to the digit set $\Om$.  Since $\la\in(0,1)$, a point in $E_\la$ may have multiple codings.

 Denote by $\Om^*$ the set of all finite words over the set $\Omega_\la$: $\Om^*:=\bigcup_{n=0}^\f\Om^n$,  where for $n=0$ we set $\Om^0:=\set{\epsilon}$ with $\epsilon$ being the empty word, and write $f_{\epsilon}$ the identity map.  Denote by $\Omega_\la^{\mathbb N}$ the set of all infinite words over the set $\Omega_\la$.

In 2004 Broomhead, Montaldi and Sidorov \cite{Bro-Mon-Sid-04} introduced the following finer  family of self-similar sets with overlaps.

 \begin{definition}\label{def:totally-self-similar}
 $E_\la$ is \emph{totally self-similar} if
 \[
 f_{\bi}(E_\la)=f_{\bi}(\De)\cap E_\la\quad\textrm{for any }\bi\in\Om^*.
 \]
 \end{definition}
 Our first result describes when  $E_\la$ is totally self-similar.

 \begin{main}
 \label{th:1}
 Let $\la\in(0,1)$. Then $E_\la$ is totally self-similar if and only if $\la=1-3^{-m}$ for some positive integer $m$.
 \end{main}

The question of spectrum was motivated by the study of non-integer base expansions   and was first initiated by Erd\H{o}s, Jo\'o and Komornik in the late 90s (see \cite{Erd-Joo-Kom-90}): Let $q> 1$ and $m\in {\mathbb N}$, define
\begin{equation*}
X_m(q):=\left\{\sum_{i=0}^n\epsilon_iq^i:\epsilon_i\in\{0,1,\cdots,m\}; ~n=0,1,\cdots\right\}.
\end{equation*}
We may arrange the elements of $X_m(q)$ into an increasing sequence as $X_m(q)$ is discrete:
\begin{equation*}
0=x_0(q,m)<x_1(q,m)<x_2(q,m)<\cdots.
\end{equation*}
Denote the spectrum
\begin{equation}\label{def:spectrumliminf}
l_m(q):=\liminf_{n\to\infty}(x_{n+1}(q,m)-x_n(q,m)).
\end{equation}

They asked for which pairs $(q,m)$ the equation $l_m(q)=0$ holds. A full answer was given by Akiyama-Komornik \cite{Aki-Kom-13} and Feng \cite{Fen-16}, which completes former partial results of Erd\"os-Komornik \cite{Erd-Kom-98} and Zaimi \cite{Zai-07}. More references 
 can be found in \cite{Erd-Joo-Kom-90,Kom-Lor-Ped-00,Sid-Sol-11}.

Since the set $\set{\sum_{i=0}^{n-1}c_i\cdot 3^i,c_i\in\set{0,\la,2-\la};~n=1,2,\cdots}$ is not necessarily discrete, so we may not arrange its elements into an increasing order as above. Therefore, we cannot define the  spectrum as in (\ref{def:spectrumliminf}). However, there do exist an equivalent form of (\ref{def:spectrumliminf}) as follows (see \cite{Aki-Kom-13}).
  \begin{align*}
l_m(q)&=\inf\left\{\left|\sum_{i=0}^n\epsilon_iq^i\right|\ne 0:\epsilon_i\in\{0,  1,\cdots, m\}-\set{0,1,\cdots, m};~n=0,1,2,\cdots\right\}\\
 &=\inf\left\{\left|\sum_{i=0}^n\epsilon_iq^i\right|\ne 0:\epsilon_i\in\{0,\pm 1,\cdots,\pm m\};~n=0,1,2,\cdots\right\}.
 \end{align*}

 We adopt a similar definition of spectrum  here.
\begin{definition}\label{def:spectrum}
For $\la\in(0,1)$,
  the \emph{spectrum} of $E_\la$ is defined by
\begin{eqnarray*}
l_\la
:=\inf\set{\left|\sum_{i=0}^{n-1}d_i\cdot 3^i\right|\ne 0:d_i\in\set{0,\pm\la,\pm(2-\la),\pm 2};~n=1,2,\cdots}.
\end{eqnarray*}
\end{definition}

It is worth mentioning that it is meaningless  to change {\bf inf} into {\bf sup} in the definition of $l_\la$. This is because if we let each $d_i$ take the value $2$, then {$\sum_{i=0}^{n-1}d_i\cdot 3^i=3^{n}-1$} tends to infinity as $n$ increases.

In the following we characterize all $\la$ such that $l_\la$ vanishes and give some properties of $l_\la$.

\begin{main}\label{th:2}\mbox{}

\begin{enumerate}[{\rm(i)}]

\item   $l_\la=0$ if and only if $\la$ is irrational.

\item $l_\la=\frac{2}{3}$ if and only if $E_\la$ is totally self-similar.
Furthermore, if $E_\la$ is not totally self-similar, then
$
0\le l_\la\le \min\set{\la, \frac{1}{2}}.
$

\item If $\la$ is rational, then $l_\la$ is computable. In particular, if $\la=\frac{m\cdot 3^n}q$ is in  lowest terms with $m=1,2$ and $n$ being a non-negative integer, then $l_\lambda=\frac mq$.

\end{enumerate}
\end{main}

By Theorems \ref{th:1} and  \ref{th:2} (ii) it follows  that for any $\la\in(0, 1)$ we have $0\le l_\la\le \min\set{\la, \frac{2}{3}}$. Furthermore, $\frac{2}{3}$ is an isolated point of the {spectrum set}
$
\La:=\set{l_\la: \la\in(0, 1)}.
$


Another line of research is about investigating all the generating iterated function systems for a self-similar set, which is of great  interest in fractal image compression (cf.~\cite{Bar-88}).

An {\it iterated function system (IFS)} is a family of contractions $\set{\phi_i(x)=\rho_iR_ix+b_i}_{i=1}^N$ ($N\ge 2$ is an integer) in ${\mathbb R}^d$, where $\rho_i\in(0,1)$ is the contraction ratio, $R_i$ is an orthogonal matrix and $b_i$ is a translation. When all $\rho_1R_1,\cdots,\rho_NR_N$ are equal, we say that the IFS $\set{\phi_i}_{i=1}^N$ is {\it homogeneous}.

It is well known that a given IFS determines a unique non-empty compact $F\subseteq {\mathbb R}^d$, which is called a {\it self-similar set}, such that $F=\bigcup_{i=1}^N\rho_iR_iF+b_i$. Is the converse true? In other words, for a given self-similar set, can we reveal the form of all its generating IFSs?

The above question was  first addressed by D.J. Feng and Y. Wang in \cite{Feng-Wang-09} and later studied by Q.R. Deng, K.S. Lau \cite{Den-Lau-13,Den-Lau-17} and Y. Yao \cite{Yao-17}. In their work, some separation property (open set condition or strong separation condition) is required a priori. By {\it strong separation condition} we mean $\phi_i(F)\cap \phi_j(F)=\emptyset$ for any two different $i,j\in\{1,\cdots,N\}$ in the IFS $\set{\phi_i(x)}_{i=1}^N$. As for the open set condition, one can refer to \cite{Mor-46} for its definition.

If we drop the separation condition assumption, the analysis of all generating IFSs for a specific self-similar set gets more involved (see e.g., \cite{Yao-Li-15}). We shall focus here on a non-trivial example: all generating IFSs for $E_\la$ when it is totally self-similar.
\begin{main}\label{th:3}
Let $\la\in(0,1)$ such that  $E_\la$ is totally self-similar.  If $g$ is an affine map and $g(E_\la) \subseteq E_\la$,  then $g=f_{\bi}$ for some $\bi\in\Om^*$.
\end{main}

There has been considerable interest in points having multiple $\beta$-expansions since it was first considered by Erd\H{o}s et al.~\cite{Erdos_Horvath_Joo_1991, Erdos_Joo_1992} (see also, Sidorov  \cite{Sid-09}).  For a systematic survey on non-integer base expansions we refer to Komornik \cite{Kom-11}.

 Given $q>1$, Dajani, Kan, Kong and Li in \cite{KKKL-3} considered expansions in base $q$  with digits set $\{0, 1, q\}$. They described the size of sets of points having finite $q$-expansions. Our question is similar to theirs but in a different setting. To be specific, we will determine the size of $E_\la$ when it is totally self-similar, and the size of the set of points having  finite triadic codings with respect to the alphabet $\Om=\set{0, \la, 2}$ as well.

For $k\in\N\cup\set{\aleph_0, 2^{\aleph_0}}$ let
\[
\u_\la^{(k)}:=\set{x\in E_\la: x\textrm{ has precisely }k\textrm{ different triadic codings} }.
\]
Then for each $x\in \u_\la^{(k)}$ there exist precisely $k$ different sequences $(d_i)\in\Om^\N$ such that $x=((d_i))_3$. In particular, for $k=1$ the set $\u_\la^{(1)}$ contains all points with a unique triadic coding.

 \begin{main}
 \label{th:4}
 Let $E_\la$ be totally self-similar with $\la=1-3^{-m}$ for some $m\in\N$. Then
 \[\dim_H E_\la=\dim_H\u_\la^{(2^{\aleph_0})}=s,\]
 where $s\in(0,1)$ satisfies $3^{1+ms}=3^{(m+1)s}+1$.
 \begin{enumerate}[{\rm(i)}]
 \item For any $k\in\N$ we have
 \[
 \dim_H\u_\la^{(k)}=t,
 \]
 where $t\in(0, s)$ satisfies $3^{1+mt}=3^{(m+1)t}+2$.

 \item $\u_\la^{(\aleph_0)}$ is countably infinite.
\end{enumerate}
 \end{main}

The rest of the paper is arranged as follows. In the next section we give equivalent conditions of totally self-similarity, and prove Theorem \ref{th:1}.
In Section \ref{s:spectrum} we investigate the spectrum of $E_\la$   and  prove Theorem \ref{th:2}. In Section \ref{s:generating IFS} we discuss the generating IFSs of $E_\la$ and establish Theorem \ref{th:3}. In Section \ref{th:5} we consider the set of points in $E_\la$ having finite or countable different codings and complete the proof of Theorem \ref{th:4}. Then we end this last section with  some open questions.


\section{When $E_\la$ is totally self-similar}\label{s:totally self-similar}

Let $\la\in(0, 1)$.  Recall that $\De=[0,1]$ is the  convex hull of $E_\la$. Set $\De_0=\De$, and for $n\ge 1$, let
\[\De_n:=\bigcup_{\bi\in\Om^n}f_{\bi}(\De).\]
Then the sequence  of sets $(\De_n)$   decreases to $E_\la$, i.e.,
\[
\De_0\supseteq\De_1\supseteq\De_2\supseteq\cdots,\quad\textrm{and}\quad \bigcap_{n=0}^\f \De_n=E_\la.
\]
The set $\De_n$ is called  the \emph{$n$-level basic set}, and each  subset $f_{\bi}(\De)$ with $\bi\in\Om^n$ is called an \emph{$n$-level basic interval}.

By a {\it hole} of $E_\lambda$ we mean a connected component in $I\backslash E_\lambda$.
Let $H:=\De\setminus\De_1=(\frac{1+\la}{3}, \frac{2}{3})$. Then $H$ is obviously a hole of $E_\la$ (see Figure \ref{fig:1} below). Set $H_0:=H$, and for $n\ge 1$, let
\[
H_n=\bigcup_{\bi\in\Om^{n}}f_{\bi}(H).
\]

 In general,  $f_{\bi}(H)$ is not necessarily   a hole of $E_\la$. For example, we can easily prove that for $\la=1/3$ the set $f_{1/3}(H)$ is not a hole of $E_{1/3}$. This is because $5/9\in H$ and $f_{1/3}\left(5/9\right)=f_{022}(0)\in E_{1/3}$. However,  when $E_\la$ is totally self-similar  we show that each $f_{\bi}(H)$ is indeed a hole of $E_\la$.

\begin{proposition}
\label{prop:equivalent-condition-totally-self-similar}
The following statements are equivalent.
\begin{enumerate}
[{\rm(i)}]
\item The set $E_\la$ is totally self-similar.

\item For any two finite words $\bi, \bj$, we have
\[
f_{\bi}(E_\la)\cap f_{\bj}(E_\la)=f_{\bi}(\De)\cap f_{\bj}(E_\la)=f_{\bi}(E_\la)\cap f_{\bj}(\De)=f_{\bi}(\De)\cap f_{\bj}(\De)\cap E_\la.
\]

\item $H_n\cap E_\la=\emptyset$ for any $n\ge 0$.

\item For any two finite words $\bi, \bj$ of the same length, we have either $f_{\bi}= f_{\bj}$ or $f_{\bi}(\De)\cap f_{\bj}(H)=\emptyset$.

\item $H_n\cap\De_{n+1}=\emptyset$ for any $n\ge 0$.

\end{enumerate}

\end{proposition}
\begin{proof}
(i) $\Rightarrow$ (ii) follows directly from Definition \ref{def:totally-self-similar}.

(ii) $\Rightarrow$ (iii). Clearly, (iii)   holds for $n=0$. Now for $n\ge 1$ let $\bi\in\Om^{n}$ and $d\in\Om$. Then by (ii) it follows that $f_{\bi}(E_\la)\cap f_d(E_\la)=f_{\bi}(\De)\cap f_d(E_\la)$. This implies
\[f_{\bi}(\De\setminus E_\la)\cap f_d(E_\la)=\emptyset.\]
Since  $H=\De\setminus\De_1\subset \De\setminus E_\la$, we obtain $f_{\bi}(H)\cap f_d(E_\la)=\emptyset$ for all $\bi\in\Om^{n}$ and all $d\in\Om$. Hence, (iii) follows from that
\[H_{n}=\bigcup_{\bi\in\Om^{n}}f_{\bi}(H)\quad \textrm{and}\quad E_\la=\bigcup_{d\in\Om}f_d(E_\la).\]

(iii) $\Rightarrow$ (iv). Let $\bi, \bj\in\Om^n$ with $n\ge 1$. Suppose $f_{\bi}\ne f_{\bj}$. We will show that $f_{\bi}(\De)\cap f_{\bj}(H)=\emptyset$. By (iii) it follows that $f_{\bi}(E_\la)\cap f_{\bj}(H)=\emptyset$. So, it suffices to prove
\begin{equation}\label{eq:(iii)-->(iv)}
f_{\bi}(\De\setminus E_\la)\cap f_{\bj}(H)=\emptyset.
\end{equation}
In view of (iii),  for each $\bk\in\Om^*\backslash \set{\epsilon}$,  note that the endpoints of $f_{\bi\bk}(H)$ and $f_{\bj}(H)$ belong to $E_\la$, the open intervals  $f_{\bi\bk}(H)$ and $f_{\bj}(H)$ are indeed different holes of $E_\la$ as they are of different length. This, together with $f_{\bi}(H)\cap f_{\bj}(H)=\emptyset$, implies that
\[
f_{\bi\bk}(H)\cap f_{\bj}(H)=\emptyset\quad\textrm{for all }\bk\in\Om^*.
\]
Hence, (\ref{eq:(iii)-->(iv)}) follows from (iii) and $f_{\bi}(\De\setminus E_\la)=\bigcup_{\bk\in\Om^*}f_{\bi\bk}(H)$.

(iv) $\Rightarrow$ (v). Let $n\ge 0$. Observe that $H_n=\bigcup_{\bi\in\Om^n}f_{\bi}(H)$ and $\De_{n+1}=\bigcup_{\bj\in\Om^n}f_{\bj}(\De_1)$. So it suffices to prove that
\begin{equation}\label{eq:(iv)--(v)}
f_{\bi}(H)\cap f_{\bj}(\De_1)=\emptyset \quad\textrm{for any }\bi, \bj\in\Om^n.
\end{equation}
If $f_{\bi}=f_{\bj}$, then (\ref{eq:(iv)--(v)})  follows from $H\cap \De_1=\emptyset$ trivially. If $f_{\bi}\ne f_{\bj}$, then (\ref{eq:(iv)--(v)}) follows from (iv) that $f_{\bi}(H)\cap f_{\bj}(\De)=\emptyset$ and $\De_1\subset \De$.

(v) $\Rightarrow$ (i). First we claim that for any $\bi\in\Om^n$ with $n\ge 0$,
\begin{equation}\label{eq:(v)--(i)-1}
f_{\bi}(H_m)=f_{\bi}(\De)\cap H_{n+m}\quad\textrm{for all }m\ge 0.
\end{equation}
We will prove this by induction on $m$. For $m=0$,  note that for $\bi\in\Om^n$ we get $f_{\bi}(\De_1)\subseteq \De_{n+1}$ and $f_{\bi}(H)\subseteq H_n$. Then by (v) it follows that
\[
f_{\bi}(\De)\cap H_n=(f_{\bi}(H)\cap H_n)\cup(f_{\bi}(\De_1)\cap H_n)=f_{\bi}(H).
\]

Now take $k\ge 0$, and assume that (\ref{eq:(v)--(i)-1}) holds for all $m\le k$ and $\bi\in\Om^n$ with $n\ge 0$. Then by the induction hypothesis it follows that
\begin{equation}\label{eq:(v)--(i)-2}
\begin{split}
f_{\bi}(H_{k+1})=f_{\bi}(\bigcup_{d\in\Om}f_d(H_k))&=\bigcup_{d\in\Om}f_{\bi d}(H_k)\\
&=\bigcup_{d\in\Om}f_{\bi d}(\De)\cap H_{n+k+1}=f_{\bi}(\bigcup_{d\in\Om}f_d(\De))\cap H_{n+k+1}\\
&=f_{\bi}(\De\setminus H)\cap H_{n+k+1}.
\end{split}
\end{equation}
Since $E_\la\subset \De_{n+1}$, by (v) we get $H_n\cap E_\la=\emptyset$ for all $n\ge 0$. So,   $f_{\bi} (H)$ is a  hole of $E_\la$. By the same argument as in the proof of (iii) $\Rightarrow$ (iv) it gives that $f_{\bi}(H)\cap H_{n+k+1}=\emptyset$. By (\ref{eq:(v)--(i)-2}) this proves (\ref{eq:(v)--(i)-1}) for $m=k+1$. Hence, (\ref{eq:(v)--(i)-1}) follows by induction.

In order to show that $E_\la$ is totally self-similar we need to prove that for any $\bi\in\Om^n$ with $n\ge 0$,
\begin{equation}\label{eq:(v)--(i)-3}
f_{\bi}(\De_m)=f_{\bi}(\De)\cap \De_{n+m}\quad\textrm{for all }m\ge 0.
\end{equation}
This will be done by induction on $m$. Clearly, (\ref{eq:(v)--(i)-3}) holds for $m=0$. Take $k\ge 0$, and we assume (\ref{eq:(v)--(i)-3}) holds for all $m\le k$ and all $\bi\in\Om^n$ with $n\ge 0$. Note that
\[
\De_k\setminus H_k=\bigcup_{\bj\in\Om^k}f_{\bj}(\De)\setminus\bigcup_{\bj\in\Om^k}f_{\bj}(H)=\bigcup_{\bj\in\Om^k}f_{\bj}(\De\setminus H)=\De_{k+1},
\]
where the second equality follows by using that $f_\bi(H)\cap f_\bj(H)=\emptyset$ for any two words $\bi, \bj\in\Om^k$.
Similarly, $\De_{n+k}\setminus H_{n+k}=\De_{n+k+1}$. Then by the induction hypothesis and (\ref{eq:(v)--(i)-1}) it follows that
\begin{align*}
f_{\bi}(\De_{k+1})=f_{\bi}(\De_k\setminus H_k)&=f_{\bi}(\De_k)\setminus f_{\bi}(H_k)\\s
&=(f_{\bi}(\De)\cap\De_{n+k})\setminus(f_{\bi}(\De)\cap H_{n+k})\\
&=f_{\bi}(\De)\cap (\De_{n+k}\setminus H_{n+k})=f_{\bi}(\De)\cap\De_{n+k+1}.
\end{align*}
This proves (\ref{eq:(v)--(i)-3}) for $m=k+1$, and hence (\ref{eq:(v)--(i)-3}) follows by induction.

Note that $\De_m$ decreases to $E_\la$ as $m\ra\f$ and $f_{\bf i}$ is continuous. Letting $m\ra\f$ in (\ref{eq:(v)--(i)-3}) yields that $f_{\bi}(E_\la)=f_{\bi}(\De)\cap E_\la$ for all $\bi\in\Om^*$.
\end{proof}

\begin{remark}\label{rem:totaly-self-similar-characterization}
\mbox{}

\begin{itemize}

\item Proposition \ref{prop:equivalent-condition-totally-self-similar} holds when replacing $E_\la$ by the attractor of a homogeneous IFS; the proof is the same. Besides, the proof also implies the equivalences (i) $\Leftrightarrow$ (ii) $\Leftrightarrow$ (iii)  in a more general setting. To be more precise, let $g_1, \ldots, g_m$ be contractive similitudes in $\R^d$, and let $F$ be the attractor of the IFS $\set{g_i}_{i=1}^m$. Denote by $\Ga$   the convex hull of $F$. Then $F$ is totally self-similar if,  and only if,   for any two words  $\bi, \bj\in\set{1,2,\ldots, m}^*$  we have
\[
f_{\bi}(F)\cap f_{\bj}(F)=f_{\bi}(\Ga)\cap f_{\bj}(F)=f_{\bi}(F)\cap f_{\bj}(\Ga)=f_{\bi}(\Ga)\cap f_{\bj}(\Ga)\cap F.
\]
This is also equivalent to that $f_{\bi}(\Ga\setminus\bigcup_{i=1}^m g_i(\Ga))\cap F=\emptyset$ for any $\bi\in\set{1,2,\ldots, m}^*$.

\item If $E_\la$ contains an interior point (which means that $\la$ is rational and $E_\la$ satisfies the open set condition by \cite{Ken-97} and \cite{Rao-Wen-98}), then $E_\la$ is not totally self-similar. This can be inferred from the following observation.

Suppose $E_\la$ contains an interior point $x$. Then there exists a word $\bi\in\Om^*$  such that $x\in f_{\bi}(\De)\subseteq E_\la$. This implies that $f_{\bi}(H)\cap E_\la=f_{\bi}(H)\ne\emptyset$. Thus  $E_\la$ is not totally self-similar according to  Proposition \ref{prop:equivalent-condition-totally-self-similar} (iii).

\item In general, the strong separation condition does not imply totally self-similar. A counterexample would be the IFS $\left\{x/3,(x+4)/{27},(x+2)/3\right\}$.

Suppose the attractor of the above IFS is $F$, then the convex hull of $F$ is $[0,1]$. It is easy to check that  $F$ is a subset of $\left[0, 1/3\right]\cup \left[ 2/3,1\right]$. Then we have  $F/3$, $(F+4)/{27}$ and $(F+2)/3$ are pairwise disjoint.
Therefore the IFS $\left\{x/3,(x+4)/{27},(x+2)/3\right\}$ satisfies the strong separation condition.
However, combining  Proposition \ref{prop:equivalent-condition-totally-self-similar} (ii) and the fact that
\[\frac{[0,1]}{3}\cap \frac{[0,1]+4}{27}\cap F\supseteq  \frac{F+4}{27}\quad\textrm{and}\quad \frac{F}{3}\cap\frac{F+4}{27}=\emptyset\]
 yields that $F$ is not totally self-similar.
\end{itemize}
\end{remark}

\begin{proof}[Proof of Theorem \ref{th:1}]
First we prove the necessity.  Denote by $\rho_k:=1-3^{-k}$ for $k\ge 0$.  Then the points $\rho_1,\rho_2,\ldots$ form a partition of the unit interval $(0,1)$. Note that $\rho_0=0$ and $\rho_k\nearrow 1$ as $k\rightarrow \infty$. So it suffices to show  that for any $k\ge 0$ the set $E_\la$ is not totally self-similar for  any $\lambda\in(\rho_k, \rho_{k+1})$.

Take $k\ge 0$ and pick $\lambda\in(\rho_k, \rho_{k+1})$.  In view of Proposition \ref{prop:equivalent-condition-totally-self-similar} (iv) it suffices to show that
  \begin{equation}\label{eq:th1-1}
  f_{0 2^k}\ne f_{\lambda 0^k}\quad \textrm{and}\quad
 f_{0 2^k}(H)\cap f_{\lambda 0^k}(\De)\ne\emptyset.
  \end{equation}
Since $\lambda>\rho_k$, a simple calculation yields
\[
f_{0 2^k}(0)= \frac{\rho_k}{3}\ne\frac{\lambda}{3}=f_{\lambda 0^k}(0).
\]
This proves the first statement of (\ref{eq:th1-1}). For the second statement   we observe that
\begin{align*}
f_{0 2^k}(H)&=\frac{\rho_k}{3}+\frac{H}{3^{k+1}}=\left(\frac{\rho_k}{3}+\frac{1+\la}{3^{k+2}}, \frac{\rho_{k+1}}{3}\right),\\
f_{\lambda 0^k}(\De)&= \frac{\lambda}{3}+\frac{\De}{3^{k+1}}=\left[\frac{\la}{3}, \frac{\la}{3}+\frac{1}{3^{k+1}}\right].
\end{align*}
Since $\lambda\in(\rho_k, \rho_{k+1})$,  it follows that
\[
\frac{\rho_k}{3}+\frac{1+\la}{3^{k+2}}<\frac{\la}{3}+\frac{1}{3^{k+1}}\quad\textrm{and}\quad \frac{\rho_{k+1}}{3}>\frac{\la}{3}.
\]
This implies $f_{02^k}(H)\cap f_{\la 0^k}(\De)\ne\emptyset$. So, (\ref{eq:th1-1}) holds, and then the necessity follows.

Now we prove the sufficiency. Let $\la=\rho_m=1-3^{-m}$. By Proposition \ref{prop:equivalent-condition-totally-self-similar} (iv) it suffices to prove that for any $n\in\N$ and for any $\bi, \bj\in\Om^n$ with $f_{\bi}\ne f_{\bj}$ we have
\begin{equation}\label{eq:thm1-2}
f_{\bi}(H)\cap f_{\bj}(\De)=\left(f_{\bi}(0)+\frac{1+\la}{3^{n+1}}, f_{\bi}(0)+\frac{2}{3^{n+1}}\right)\cap \left[f_{\bj}(0), f_{\bj}(0)+\frac{1}{3^n}\right]=\emptyset.
\end{equation}
Clearly, (\ref{eq:thm1-2}) holds if $f_{\bi}(0)+\frac{2}{3^{n+1}}\le f_{\bj}(0)$ or $f_{\bi}(0)+\frac{1+\la}{3^{n+1}}\ge f_{\bj}(0)+\frac{1}{3^n}$. Therefore,   (\ref{eq:thm1-2}) follows once we can prove
\begin{equation}\label{eq:thm1-3}
3^n|f_{\bi}(0)-f_{\bj}(0)|\ge \frac{2}{3}\quad\textrm{for any }\bi, \bj\in\Om^n\textrm{ with }  f_{\bi}\ne f_{\bj};  n\in\N.
\end{equation}

Let $n\in\N$ and take $\bi=i_1\ldots i_n, \bj=j_1\ldots j_n\in\Om^n$ such that $f_{\bi}(0)\ne f_{\bj}(0)$. Note that
\begin{equation}\label{eq:thm1-4}
3^n|f_{\bi}(0)-f_{\bj}(0)|=\left|\sum_{k=1}^n (i_k-j_k)3^{n-k}\right|=:\left|\sum_{k=0}^{n-1} c_k 3^{k}\right|,
\end{equation}
where for each $0\le k<n$,
\[
c_k:=i_{n-k}-j_{n-k}\in\Om-\Om=\set{0,\pm\rho_m, \pm(2-\rho_m), \pm 2}.
\]

We will prove (\ref{eq:thm1-3}) by adopting an idea from \cite{Kom-Lor-Ped-00}. In view of (\ref{eq:thm1-4}) we assume on the contrary to (\ref{eq:thm1-3}) that there exists a sequence of integers  $c_0,c_{1},\cdots,c_{n-1}\in\set{0,\pm\rho_m, \pm(2-\rho_m), \pm 2}$ such that
\begin{equation}\label{2a}
0<\left|\sum_{j=0}^{n-1}c_j3^j\right|<\frac 23.
\end{equation}
Choose the sequence $(c_j)$ such that the sequence
\begin{equation}\label{minimal}
|c_{n-1}|,\cdots,|c_0|\;{\textrm{is lexicographically minimal}}
\end{equation}
among all sequences satisfying (\ref{2a}). Without loss of generality we can assume that  $c_{n-1}>0$.

If $c_{n-1}=2$ or $2-\rho_m$, then (\ref{eq:thm1-3}) follows from (\ref{eq:thm1-4}) and the following calculation:
 \begin{align*}
 \left|\sum_{j=0}^{n-1}c_j 3^j\right|&\ge (2-\rho_m)3^{n-1}-2\sum_{j=0}^{n-2}3^j=1+\frac{3^{n-1}}{3^m}>\frac{2}{3}.
 \end{align*}

It remains to show the case $c_{n-1}=\rho_m$. Observe that if $c_i=\rho_m$ for some $i\ge m$, then at least one of its following $m$ coefficients $c_{i-1},\cdots,c_{i-m}$ should be positive. Otherwise, we could change
$c_i3^i+c_{i-1}3^{i-1}+\cdots+c_{i-m}3^{i-m}$ to $(c_i-\rho_m)3^i+(c_{i-1}+2)3^{i-1}+\cdots+(c_{i-m}+2)3^{i-m}$ by using $\rho_m=\sum_{j=1}^m2\cdot{3^{-j}}$, contradicting  (\ref{minimal}).

Now, writing ${ n}=pm+\ell$ with {$p\ge 0$} and $0\le \ell<m$, then we have
 \begin{align*}
\sum_{j=0}^{n-1}c_j3^j&\ge (\rho_m\cdot 3^{n-1}-2\cdot 3^{n-2}-\cdots-2\cdot 3^{n-m})+(\rho_m\cdot 3^{n-m-1}-2\cdot 3^{n-m-2}-\cdots-2\cdot 3^{n-2m})\\
&+\cdots+(\rho_m\cdot 3^{n-pm-1}-2\cdot 3^{n-pm-2}-\cdots-2\cdot 3^0)\\
&=2\cdot 3^{n-m-1}+2\cdot 3^{n-2m-1}+\cdots+2\cdot 3^{n-pm-1}+2(3^{-1}+
\cdots+3^{\ell-m-1})\ge \frac 23,
 \end{align*}
contradicting (\ref{2a}). { Therefore, this proves (\ref{eq:thm1-3}), and completes the proof.}
\end{proof}

\section{The spectrum of $E_\la$}\label{s:spectrum}


 Recall {from Definition \ref{def:spectrum}} that the spectrum of $E_\la$ is given by $l_\la :=\inf A_\la, $ where
\begin{equation*}
A_\la:=\set{\left|\sum_{i=0}^{n-1}d_i\cdot 3^i\right|\ne 0,\,d_i\in\set{0,\pm\la,\pm (2-\la),\pm 2};~n=1,2,\cdots}.
\end{equation*}
{ Observe that $\set{0,\pm \la, \pm(2-\la), \pm 2}=\Om-\Om$. Then} we can
 rewrite $A_\la$ as
 \begin{equation*}A_\la=\set{3^n|f_{\bi}(0)-f_{\bj}(0)|: \bi, \bj\in\Om^n\textrm{ with }f_{\bi}\ne f_{\bj};~n=1,2,\cdots}.
 \end{equation*}

For any two different words $\bi, \bj$ of the same length $n$, in view of (\ref{eq:fd-similitude}), it is clear that $f_{\bi}=f_{\bj}$ if and only if $f_{\bi}(0)=f_{\bj}(0)$. So, the scaled distance $3^n|f_{\bi}(0)-f_{\bj}(0)|$ describe the closeness of the two maps $f_{\bi}$ and $f_{\bj}$, which turns out to reflect the structure of $E_\la$.

We first  consider some examples where the spectrum $l_\la$ can be explicitly determined. {We need the following result by Rao and Wen \cite{Rao-Wen-98}.
\begin{lemma}(\cite{Rao-Wen-98})
\label{lem:R-W}
Let $\la\in(0, 1)$ and let $q$ be a positive integer. If $d\in\set{1,2,\ldots, q}\cap q\cdot A_\la$, then
\[
g^{n}_{2q}(d)\in q\cdot A_\la
\]
for any $n\ge 1$, where $g_{2q}: \mathbb Z\setminus\set{0}\to\mathbb Z\setminus\set{0}$ is defined by
\begin{equation*}
g_{2q}(x)=\left\{\begin{array}{ll}
\displaystyle\frac{2q+x}{3},&\quad{\textrm{if}}\quad 3|(2q+x),\\
\displaystyle\frac{2q-x}{3},&\quad{\textrm{if}}\quad 3|(2q-x),\\
\displaystyle\frac{x}{3},&\quad{\textrm{if}}\quad 3|x.
\end{array}\right.
\end{equation*}
\end{lemma}

\begin{lemma}\label{spectrum:m=12}
Let $\la=\frac{m\cdot 3^k}q$ be in lowest terms with  $k$ being a non-negative integer, then
\[\frac{1}{q}\le l_\lambda\le\frac mq.\] In particular, for $m=1,2$  we have $l_\lambda=\frac mq$.
\end{lemma}
\begin{proof} Suppose $\la=\frac{m\cdot 3^k}{q}=\frac pq<1$  is a rational in  lowest terms, we will first prove that $l_\lambda\le\frac mq$.

The case $k=0$ is obvious if we take $n=1$ and $d_0=\la$ in the definition of $l_\la$. Then it suffices to consider  the case $k\ge1$.
%
%
%
Note that   $p\in \{1,2,\cdots,q\}$ and $p=q\cdot \la\in qA_\la$. Then by Lemma \ref{lem:R-W} it follows that
 \[g_{2q}^k(p)=\frac{p}{3^k}\in qA_\la.\]
 So there exists   $x\in A_\la$ such that
 \[x=\frac{p}{q\cdot 3^k}=\frac{\la}{3^k}=\frac{m}{q}.\] This implies $l_\la\le x=\frac mq$.

On the other hand, note that for any positive integer $n$ and
any words ${\bf i,\bf j}\in \Omega^n$ with $f_{\bf i}(0)\ne f_{\bf j}(0)$, we have
\begin{equation}\label{ge}
3^n\left|f_{\bf i}(0)-f_{\bf j}(0)\right|=\left|\sum_{i=0}^{n-1}d_i\cdot 3^i\right| =\frac 1q\left|\sum_{i=0}^{n-1}\left(q d_i\right)\cdot 3^i\right|,
\end{equation}
with  $q d_i\in q\set{0,\pm\frac{p}{q}, \pm(2-\frac{p}{q}), \pm 2}=\{0,\pm p,\pm 2q,\pm(2q-p)\}$.
It follows that $\left|\sum_{i=0}^{n-1}\left(q d_i\right)\cdot 3^i\right|$ is a positive integer or a positive even integer, depending on whether $m=1$ or $2$. Thus $l_\lambda=\inf\{3^n\left|f_{\bf i}(0)-f_{\bf j}(0)\right|: \bi, \bj\in\Om^n\textrm{ with }f_{\bi}\ne f_{\bj};~n=1,2,\cdots\}
\ge\frac mq$ with $m=1$ or $2$. The proof is finished.
\end{proof}}

In the following lemma we determine $l_\la$ when $E_\la$ is totally self-similar.
\begin{lemma}\label{lem:spectrum-2/3}
Let $\la\in(0,1)$. If $E_\la$ is totally self-similar, then $l_\la=\frac{2}{3}$.
\end{lemma}
\begin{proof}
Take $\la\in(0, 1)$ such that $E_\la$ is totally self-similar. By Theorem \ref{th:1} there exists $m\in\N$ such that
  $\la=1-3^{-m}$. Then by (\ref{eq:thm1-3}) it follows that   the spectrum
\[
l_\la=\inf\set{3^n|f_{\bi}(0)-f_{\bj}(0)|: \bi, \bj\in\Om^n\textrm{ with } f_{\bi}\ne f_{\bj};~n=1,2,\cdots}\ge \frac{2}{3}.
\]
On the other hand,   take $\bi=02^{m-1}$ and $\bj=\la 0^{m-1}$. Then  one can verify that
\[
3^m|f_{\bi}(0)-f_{\bj}(0)|=\frac{2}{3}.
\]
This proves $l_\la=\frac{2}{3}$.
\end{proof}

Now we prove that $l_\la=\frac{2}{3}$ is also the sufficient condition for  $E_\la$ to be totally self-similar.
\begin{lemma}
\label{lem:spectrum-upper bound}
Let $\la\in(0, 1)$. If $E_\la$ is not totally self-similar, then $0\le l_\la\le\min\set{\la, \frac{1}{2}}$.
\end{lemma}
\begin{proof}
Clearly, for $\la\in(0, 1)$ we have $l_\la\le \la$ since $3|f_\la(0)-f_0(0)|=\la$.
Let $\la\in(0, 1)$ such that $E_\la$ is not totally self-similar. It suffices to prove that $l_\la\le \frac{1}{2}$.
By Theorem \ref{th:1}   we have $\la\ne \rho_m=1-3^{-m}$ for all $m\in\N$. Write
\[
\la=((\sla_k))_3=\sum_{k=1}^\f\frac{\sla_k}{3^k}
\]
such that $(\sla_k)\in\set{0, \pm 2}^\f$ is the greedy triadic expansion of $\la$.
We distinguish  the following three cases: (A) $(\sla_k)$ contains zeros but does not end with $0^\f$; (B) $(\sla_k)$ ends with $0^\f$; (C) $(\sla_k)$  contains no zeros.

  Case (A). $(\sla_k)$ contains zeros  but does not end  with $0^\f$.  Then there exists $N\in\N$ such that $\sla_N=0$ and $\sla_{N+1}\in\set{-2, 2}$.   Take $\bi=i_1\ldots i_N, \bj=j_1\ldots j_N\in\Om^N$ such that $i_1=\la, j_1=0$, and for $2\le k\le N$,
  \begin{equation}\label{eq:i-j-construction}
  \left\{\begin{array}{lll}
  i_k=0, ~j_k=\sla_{k-1}&\quad\textrm{if}& \sla_{k-1}\in\set{0, 2},\\
i_k=2, ~ j_k=0&\quad\textrm{if}& \sla_{k-1}=-2.
  \end{array}\right.
  \end{equation}
  Then $j_k-i_k=\sla_{k-1}$ for any $2\le k\le N$.
Therefore, by using $\sla_N=0$ and $|\sla_{N+1}|=2$ it follows that
  \begin{align*}
  3^N|f_{\bi}(0)-f_{\bj}(0)|&=3^N\left|\sum_{k=1}^N\frac{i_k-j_k}{3^k}\right|=3^N\left|\frac{\la}{3}-\sum_{k=2}^N\frac{\sla_{k-1}}{3^k}\right|\\
  &=3^{N-1}\left|\sum_{k=1}^\f\frac{\sla_k}{3^k}-\sum_{k=1}^{N-1}\frac{\sla_k}{3^k}\right|\\
  &=\left|\sum_{k=1}^\f\frac{\tilde\la_{N+k-1}}{3^k}\right|\le \sum_{k=2}^\f\frac{2}{3^k}=\frac{1}{3},
  \end{align*}
and
  \[
   3^N|f_{\bi}(0)-f_{\bj}(0)|=\left|\sum_{k=1}^\f\frac{\tilde\la_{N+k-1}}{3^k}\right|\ge \frac{2}{3^2}-\sum_{k=3}^\f\frac{2}{3^k}=\frac{1}{9}>0.
  \]
  This implies $l_\la\le  \frac{1}{3}<\frac{1}{2}$ if the greedy triadic expansion of $\la$ contains zeros  but does not end with $0^\f$.

  Case (B). $(\sla_k)$ ends with $0^\f$. Since $\la\in(0,1)$ and $\la\ne \rho_m$ for any $m\in\N$, this gives $(\sla_k)\ne 2^n 0^\f$ for any $n\in\N$. So there exists $N\in\N$ such that
  \[
  \sla_N\sla_{N+1}\ldots\in\bigcup_{m=1}^\f\set{0(-2)^m 0^\f, 02^m 0^\f, 2(-2)^m 0^\f, (-2)2^m 0^\f}.
  \]

  If $\sla_N\sla_{N+1}\ldots\in\bigcup_{m=1}^\f\set{0(-2)^m 0^\f, 02^m 0^\f}$, then by the same argument as in Case (A) we can prove $l_\la\le 1/3<1/2$. Now suppose $\sla_N\sla_{N+1}\ldots\in\bigcup_{m=1}^\f\set{ 2(-2)^m 0^\f, (-2)2^m 0^\f}$.

  Without loss of generality we assume $\sla_N\sla_{N+1}\ldots=2(-2)^m 0^\f$ for some $m\in\N$. Take $\mathbf i=i_1\ldots i_N, \mathbf j=j_1\ldots j_N\in\Om^N$ such that $i_1=\la, j_1=0$ and the blocks $i_2\ldots i_N, j_2\ldots j_N$
  satisfy (\ref{eq:i-j-construction}). By the same argument as in Case (A) one can show that
  \begin{align*}
  3^N|f_{\mathbf i}(0)-f_{\mathbf j}(0)|=\left|\sum_{k=1}^\f\frac{\sla_{N+k-1}}{3^k}\right|=\frac{2}{3}-\sum_{k=2}^{m+1}\frac{2}{3^k}\subseteq\left[\frac{1}{3}, \frac{4}{9}\right].
  \end{align*}
  This implies $l_\la\le 4/9<1/2$.

  Case (C). $(\sla_k)$  contains no zeros. Then  $(\sla_k)\in\set{-2, 2}^\N$. Since $\la\in(0, 1)$, there exists an integer $n$ such that $\sla_n\ne \sla_{n+1}$. So $\sla_n=-\sla_{n+1}$. Let
  \[
  \mathcal N:=\set{n\in\N: \sla_n=- \sla_{n+1}=- \sla_{n+2}}.
  \]
  We point that the set $\mathcal N$ might be empty, and in this case the sequence $(\sla_k)$ ends with $(2(-2))^\f$.

 \begin{itemize}
 \item  If $\mathcal N\ne \emptyset$, then take $N\in\mathcal N$.  Let $\bi=\la\, i_2\ldots i_N, \bj=0 \,j_2\ldots j_N\in\Om^N$ such that the words $i_2\ldots i_N$ and $j_2\ldots j_N$ satisfy (\ref{eq:i-j-construction}).   Then by a similar argument as in Case (A) it follows that
  \[
   0<3^N|f_{\bi}(0)-f_{\bj}(0)|=\left|\sum_{k=1}^\f\frac{\tilde\la_{N+k-1}}{3^k}\right|\le \frac{2}{3}-\frac{2}{3^2}-\frac{2}{3^3}+\sum_{k=4}^\f\frac{2}{3^k}=\frac{11}{27}<\frac{1}{2}.
  \]

 \item  If $\mathcal N=\emptyset$, then there exits $N\in\N$ such that $\sla_N\sla_{N+1}\ldots=(2(-2))^\f$. Again, let  $\bi=\la\, i_2\ldots i_N, \bj=0 \,j_2\ldots j_N\in\Om^N$ such that the words $i_2\ldots i_N$ and $j_2\ldots j_N$ satisfy (\ref{eq:i-j-construction}). By a similar argument we obtain
 \[
  3^N|f_{\bi}(0)-f_{\bj}(0)|=\left|\sum_{k=1}^\f\frac{\tilde\la_{N+k-1}}{3^k}\right|=2\sum_{k=1}^\f\frac{(-1)^{k+1}}{3^k}=\frac{1}{2}.
 \]

    \end{itemize}

 Hence, for any $\la\in(0, 1)\setminus\set{\la_m: m\in\N}$ we have $l_\la\le \frac{1}{2}$. This completes the proof.
\end{proof}
\begin{remark} The upper bound in Lemma \ref{lem:spectrum-upper bound} is optimal. One can find values of $\la$ such that $l_\la=\min\{\la,\frac 12\}$.  A simple example would be $\la=\frac 12$ (see Lemma \ref{spectrum:m=12}).
\end{remark}

In order to prove Theorem \ref{th:2}, we still need the following results from \cite{Lau-Ngai-99} and \cite{Ngai-Wang-01}.

Let $\Phi=\{\phi_i\}_{i=1}^N$ be a family of contractive similitudes on ${\mathbb R}$ of the form $\phi_i(x)=\rho x+b_i,i=1,\cdots,N$ with $b_1<b_2<\cdots<b_N$. The following definition was first introduced in \cite{Lau-Ngai-99} in a slightly different but equivalent form (see also \cite{Zer-96}).

\begin{definition}We say that $\Phi$ satisfies the {\it weak separation condition} if there exists a constant $C>0$ such that for each positive integer $n$ and any two indices ${\bf i,j}\in
\{1,\cdots,N\}^n$, either $\phi_{\bf i}(0)=\phi_{\bf j}(0)$ or $\rho^{-n}|\phi_{\bf i}(0)-\phi_{\bf j}(0)|\ge C.$
\end{definition}

\begin{remark}\label{rem:weak separation} By the above definition, we have $l_\lambda>0$ if and only if the IFS $\{x/3,(x+\lambda)/3,(x+2)/3\}$ satisfies the weak separation condition.
\end{remark}

The next definition  is adopted from \cite{Fen-03} and is equivalent to the more general definition in \cite{Ngai-Wang-01}.
\begin{definition}\label{def:finite type condition}We say that $\Phi$ satisfies the {\it finite type condition} if there is a finite set $\Gamma$ such that for each positive integer $n$ and  any two indices ${\bf i,j}\in
\{1,\cdots,N\}^n$, either $\rho^{-n}|\phi_{\bf i}(0)-\phi_{\bf j}(0)|>\frac{b_N-b_1}{1-\rho}$ or $\rho^{-n}|\phi_{\bf i}(0)-\phi_{\bf j}(0)|\in\Gamma$.
\end{definition}
It follows from $3|f_\la(0)-f_0(0)|=\la$ and the definition of spectrum that $l_\la\le \la<\frac{2/3-0}{1-1/3}$. So
  if the IFS $\{x/3,(x+\lambda)/3,(x+2)/3\}$ satisfies the finite type condition, then by Definition \ref{def:finite type condition} there exists a finite set $\Gamma$ such that
\begin{eqnarray*}l_\la&=&\inf\set{3^n|f_{\bi}(0)-f_{\bj}(0)|: \bi, \bj\in\Om^n\textrm{ with } f_{\bi}\ne f_{\bj};~n=1,2,\cdots}\\
&=&\inf\set{3^n|f_{\bi}(0)-f_{\bj}(0)|\in \Gamma^{'}:\Gamma^{'}\,\,{\textrm{is a subset of}}\,\,\Gamma},
\end{eqnarray*}
which means $l_\la$ is computable. Thus to prove the first part of   Theorem \ref{th:2} (iii), we need a sufficient condition for the IFS $\{x/3,(x+\lambda)/3,(x+2)/3\}$ to be of finite type condition, which can be deduced directly from    \cite[Theorem 2.5]{Ngai-Wang-01}.

\begin{theorem}[A simplification of   Theorem 2.5 in \cite{Ngai-Wang-01}] \label{ngai-wang}Let $\Phi=\{\phi_i\}_{i=1}^N$ be an IFS in ${\mathbb R}$ having the form $\phi_i(x)=\rho x+b_i,i=1,\cdots,N$ with $b_1<b_2<\cdots<b_N$, where $ 1/\rho$ is a pisot number. Assume that $\{b_1,\cdots,b_N\}\subset r{\mathbb Z}[1/\rho]$ for some real number $r$, then the IFS is of finite type condition.
\end{theorem}
 Applying Theorem \ref{ngai-wang} yields the following lemma.
 \begin{lemma}\label{computable}If $\la$ is rational, then the spectrum $l_\la$ is computable.
 \end{lemma}

Now we are ready to prove Theorem \ref{th:2}.
\begin{proof}[Proof of Theorem \ref{th:2}]
By Lemmas \ref{spectrum:m=12}, \ref{lem:spectrum-2/3} ,\ref{lem:spectrum-upper bound}  and \ref{computable}, it suffices to prove (i).
 If $\lambda$ is rational, suppose $\lambda$ is a rational of the form $\frac pq$ in lowest terms. Then by Lemma \ref{spectrum:m=12} it follows that  $l_\lambda\ge \frac 1q>0$.

Conversely, suppose $\lambda$ is irrational. Then  B. Solomyak and P. Shmerkin proved in Theorem 1.6 of \cite{Hoc-14}  that $\dim_HE_\lambda=1$.
Suppose on the contrary that  $l_\lambda>0$. Then by Remark \ref{rem:weak separation} it follows that  the IFS $\{x/3,(x+\lambda)/3,(x+2)/3\}$ satisfy the weak separation condition. Note that $\dim_HE_\lambda=1$, using Theorem 3 of
Zerner \cite{Zer-96}  yields that $E_\lambda$ contains  interior points. However, this can only happen when $\lambda$ is rational by  Lemma 4  in \cite{Ken-97}. Therefore, $l_\la=0$, and we completes   the proof of (i).
\end{proof}

\section{Generating IFSs for $E_\la$}\label{s:generating IFS}
In this section we will discuss all generating IFSs for $E_\lambda$ when $E_\lambda$ is totally self-similar, and prove Theorem \ref{th:3}.
Fix $\la\in(0, 1)$ such that $E_\la$ is totally self-similar. Then by Theorem \ref{th:1} it follows that $\la=1-3^{-m}$ for some $m\in\N$. From now on we fix this $\la=1-3^{-m}$.

Let $g$ be an affine map on $\R$, write
\[
g(x):=\mu x+b\quad\textrm{with}\quad |\mu |>0 \textrm{ and }b\in\R.
\]
Accordingly, denote by $g(E_\la)=\mu E_\la+b:=\set{\mu x+b: x\in E_\la}$.
Suppose $g(E_\la)\subseteq E_\la$. We are going to show that $g$ must be of the form $g=f_{\mathbf i}$ for some $\mathbf i\in\Om^*$.

Since the diameters of $g(E_\la)$ and $E_\la$ are $|\mu|$ and $1$ respectively, by using $g(E_\la)\subseteq E_\la$ it follows that $|\mu|\le 1$. If $|\mu|=1$, i.e., $\mu=\pm 1$, then by using the asymmetry of $E_\la$ (see Figure \ref{fig:1}) we conclude that $\mu=1$ and $b=0$. So, $g=f_{\epsilon}$ is the identity map. Excluding this trivial case, in the following we assume $0<|\mu|<1$.
Since the smallest and the largest elements of $E_\la$ are $0$ and $1$, respectively, it follows that
\begin{equation}\label{eq:b}
\left\{\begin{array}{lll}
0\le b\le 1-\mu&\quad\textrm{if}\quad& \mu\in(0,1),\\
-\mu\le b\le 1&\quad\textrm{if}\quad & \mu\in(-1, 0).
\end{array}\right.
\end{equation}

We first introduce two lemmas that are frequently used in the proof of Theorem \ref{th:3}.  The first one follows from the geometrical structure of $E_\la$ (see Figure \ref{fig:1}).

\begin{figure}[H]
  \centering
  \includegraphics[width=8cm]{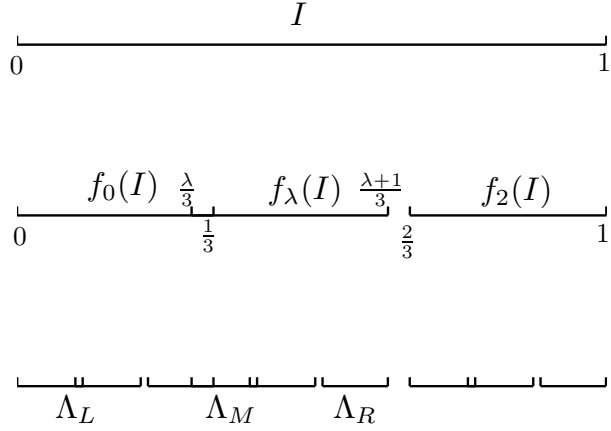}\\
  \caption{The first three  level basic intervals of $E_\la$ with $\la=1-3^{-2}=8/9$. Here $\La_L:=f_{00}(\De)\cup f_{0\la}(\De)$, $\La_M:=f_{02}(\De)\cup f_{\la 0}(\De)\cup f_{\la\la}(\De)$ and $\La_R:=f_{\la 2}(\De)$. Then the diameters of $\La_L, \La_M$ and $\La_R$ are $\frac{\la+1}{9}, \frac{4\la-1}{9}$ and $\frac{1}{9}$, respectively.}\label{fig:1}
\end{figure}

\begin{lemma}
\label{lem:geometric structure-1}
If $g(E_\la) \subseteq E_\la$, then either   $g(E_\la)\subseteq f_0(E_\la)\cup f_\la(E_\la)$ or $g(E_\la)\subseteq f_2(E_\la)$.
\end{lemma}
\begin{proof}
In view of  Figure \ref{fig:1} it suffices to prove that $g(E_\la)$ can not intersect both $f_{\la}(E_\la)$ and $f_2(E_\la)$. Suppose  on the contrary that
\begin{equation}\label{eq:0407-1}
g(E_\la)\cap f_{\la}(E_\la)\ne\emptyset\quad\textrm{and}\quad g(E_\la)\cap f_2(E_\la)\ne\emptyset.
\end{equation}
Then both $E_\lambda$ and $[0,1]$ are the attractor  of the IFS $\left\{f_0,f_\la,f_2,g\right\}$, so they should be the same, leading to a contradiction.
\end{proof}

For a non-empty compact set $A\subseteq \R$ we denote by $A_{\min}$ and $A_{\max}$ the   smallest  and the largest elements of $A$, respectively.
\begin{lemma}\label{lem:gemometric structure-2}
Suppose $A\subseteq E_\la$ and  $g(A)  \subseteq E_\la$.
\begin{enumerate}
\item[{\rm{(i)}}]  If $g(A)_{\max}\le \frac{1}{3}$, then $g(A)\subseteq f_0(E_\la)$;
\item[{\rm{(ii)}}] If $\frac \lambda3\le g(A)_{\min}\le \frac{1+\la}{3}$, then
$g(A)\subseteq f_\la(E_\la)$;
\item[{\rm{(iii)}}] If $g(A)_{\min}>\frac{1+\la}{3}$, then $g(A)\subseteq f_2(E_\la)$.
\end{enumerate}
\end{lemma}
\begin{proof}
Suppose
$A\subseteq E_\la$ and $g(A)_{\max}\le \frac 13$. Since  $E_\la$ is totally self-similar, by Proposition \ref{prop:equivalent-condition-totally-self-similar} (ii)  it follows that
\begin{align*}
g(A)&\subseteq (f_0(E_\la)\cup f_{\la}(E_\la))\cap f_{0}(\De)\\
&=(f_0(E_\la)\cap f_0(\De))\cup (f_{\la}(E_\la)\cap f_0(\De))=f_0(E_\la)\cup (f_{0}(E_\la)\cap f_{\la}(\De))=f_0(E_\la),
\end{align*}
This proves (i).
For (ii) we note by Lemma \ref{lem:geometric structure-1} that $g(A)\subseteq (f_0(E_\la)\cup f_\la(E_\la))\cap f_\la(\De)$. Then
by Proposition \ref{prop:equivalent-condition-totally-self-similar} (ii) it follows that
\[
g(A)\subseteq (f_0(E_\la)\cap f_\la(E_\la))\cup f_\la(\De)=(f_0(\De)\cap f_\la(E_\la))\cup f_\la(E_\la)=f_\la(E_\la).
\]
This establishes (ii).
Finally, in view of  Figure \ref{fig:1}, (iii) follows from Lemma \ref{lem:geometric structure-1}.
\end{proof}

Let $g$ be an affine map with $g(E_\la)=\mu E_\la+b\subseteq E_\la$.
First we show that the contraction ratio  $\mu$ must be of the form $3^{-n}$ for some positive integer $n$.

\begin{proposition}\label{prop:positive}
 If $g(E_\la)=\mu E_\la+b\subseteq E_\la$ with $0<|\mu|<1$, then $\mu=3^{-n}$ for some $n\in\N$.
\end{proposition}

The proof of Proposition \ref{prop:positive} is divided into several lemmas. First, we deduce from Lemma \ref{lem:gemometric structure-2} the following.

\begin{lemma}\label{lem:negative-1}
 Suppose $g(E_\la)=\mu E_\la+b\subseteq f_0(E_\la)\cup f_\la(E_\la)$.
\begin{enumerate}[{\rm(i)}]
\item If $\mu>0$, then
\begin{align*}
b\le \frac{1-\mu}3\quad&\Longrightarrow\quad \mu E_\la+3b\subseteq E_\la;\\
b\ge \frac{1-\mu}3 \quad&\Longrightarrow\quad \mu E_\la+3b+2\mu-\la\subseteq E_\la\textrm{ and }\mu E_\la+3b+\la\mu-\la\subseteq E_\la.
\end{align*}
\item  If $\mu<0$, then
\begin{align*}
b\le \frac{\la-2\mu}3\quad&\Longrightarrow\quad \mu E_\la+3b+2\mu\subseteq E_\la;\\
 b\ge \frac{\la-2\mu}3\quad &\Longrightarrow\quad \mu E_\la+3b-\la\subseteq E_\la\textrm{ and }\mu E_\la+3b+\la\mu-\la\subseteq E_\la.
 \end{align*}
\end{enumerate}
\end{lemma}
\begin{proof} We only prove (i), the proof of (ii) is almost the same.
If $b\le  \frac {1-\mu}3$, then
\[g(f_0(E_\la))_{\max}=g\left(\frac{E_\la}{3}\right)_{\max}=\frac{\mu}{3}+b\le \frac{1}{3}.\]
By Lemma \ref{lem:gemometric structure-2} (i) we have $g(f_0(E_\la))\subseteq f_0(E_\la)$, i.e.,
\[
\mu\frac{E_\la}{3}+b\subseteq \frac{E_\la}{3},
\]
which gives $\mu E_\la+3b\subseteq E_\la$. This proves (i).

For the second part of (i), if $b\ge  \frac{1-\mu}3$, then
\begin{align*}
g(f_2(E_\la))_{\min}>g(f_\la(E_\la))_{\min}&=g\left(\frac{E_\la+\la}{3}\right)_{\min}=\mu\frac{\la}{3}+b\\
&\ge \frac{\lambda\mu+1-\mu}{3}=\frac{\la+(1-\la)(1-\mu)}{3}>\frac{\la}{3}.
\end{align*}
Note that $g(E_\la)\subseteq f_0(E_\la)\cup f_\la(E_\la)$. So by Lemma \ref{lem:gemometric structure-2} (ii) it follows that $g(f_\la(E_\la))\subseteq f_\la(E_\la)$ and $g(f_2(E_\la))\subseteq f_\la(E_\la)$. Then
\[
\mu\frac{E_\la+\la}{3}+b\subseteq \frac{E_\la+\la}{3}\quad\textrm{and}\quad \mu\frac{E_\la+2}{3}+b\subseteq \frac{E_\la+\la}{3}.
\]
Therefore,  $\mu E_\la+3b+\mu\la-\la\subseteq E_\la$ and $\mu E_\la+3b+2\mu-\la\subseteq E_\la$. This establishes (i).
\end{proof}

\begin{lemma}\label{lem:step1-1}
If $g(E_\la)=\mu E_\la+b\subseteq E_\la$, then $0<|\mu|\le 1/3$.
\end{lemma}
\begin{proof}We split the proof into two cases: $b=0$ and $b\ne 0$. The former is simpler than the latter.

Case 1. $b=0$.
Since $1\in E_\la$, we have $\mu=\mu\cdot 1\in E_\la\subseteq [0,1]$. Note that $0<|\mu|<1$. Then $\mu\in(0,1)$. Combining Lemma \ref{lem:geometric structure-1} and $g(E_\la)_{\min}=0$ yields $g(E_\la)\subseteq f_0(E_\la)\cup f_\la(E_\la)$. We will prove that $\mu\in (0,1/3]$.

In view of Figure \ref{fig:1},  the unique largest gap in $E_\la$ is between the basic intervals $f_\la(\De)$ and $f_2(\De)$, and it has length $L:=(1-\la)/{3}$.  Therefore,  the largest gap in $f_0(E_\la)\cup f_\la(E_\la)$ has length $L/3$.  By using $\mu\in(0,1)$ this implies that the largest gap in $g(f_0(E_\la)\cup f_\la(E_\la))=\mu(f_0(E_\la)\cup f_\la(E_\la))$ has length strictly smaller than $L/3$. Observe  that $g(f_0(E_\la)\cup f_\la(E_\la))_{\min}=0$ and  the gap between the two basic intervals $f_{0\la}(\De)$ and $f_{02}(\De)$  has length $L/3$. Hence,
\begin{equation}\label{eq:propb0-1}
g(f_0(E_\la)\cup f_\la(E_\la))\subseteq f_{00}(E_\la)\cup f_{0\la}(E_\la)=f_0(f_0(E_\la)\cup f_\la(E_\la)).
\end{equation}
Since $g(x)=\mu x$ with $\mu\in(0,1)$ and $f_0(x)=x/3$, by (\ref{eq:propb0-1}) it follows that
$\mu\in(0,1/3]$.

Case 2. $b\ne 0$.
Suppose on the contrary that $|\mu|\in(1/3, 1)$. Since $g(E_\la)=\mu E_\la+b\subseteq E_\la$,  by Lemma \ref{lem:geometric structure-1} it follows that   $g(E_\la)\subseteq f_0(E_\la)\cup f_\la(E_\la)$. We split two cases: $\mu>0$ and $\mu<0$.

(i) $\mu>0$.
We claim that
\begin{equation}\label{eq:412-0}
3b<1-\mu.
\end{equation}

If $3b\ge 1-\mu$, then  by Lemma \ref{lem:negative-1} (i) it follows that
$
\mu E_\la+3b+2\mu-\la\subseteq E_\la.
$
Observe that
\[
(\mu E_\la+3b+2\mu-\la)_{\min}=3b+2\mu-\la\ge 1+\mu-\la\ge \mu>\frac{1}{3}>\frac{\la}{3}.
\]
Also note that $g(E_\la)\subseteq f_0(E_\la)\cup f_\la(E_\la)$. Then by   Lemma \ref{lem:gemometric structure-2}   it follows that
\[
\mu E_\la+3b+2\mu-\la\subseteq f_\la(E_\la).
\]
This implies that the contraction ratio $\mu$ is less than or equal to $1/3$, leading to a contradiction with our hypothesis $\mu\in(1/3,1)$. This proves the claim.

By (\ref{eq:412-0}) and Lemma \ref{lem:negative-1} (i) it follows that
$\mu E_\la+3b\subseteq E_\la.$ By the same argument as above with $b$ replaced by $3b$ one can prove that $3^2 b<{1-\mu}$. Again, by Lemma \ref{lem:negative-1} (i) we can deduce that $\mu E_\la+3^2b\subseteq E_\la$.
Iterating the above argument infinitely gives that
\[\mu E_\la+3^p b\subseteq E_\la\quad\textrm{for all }p\in\N.\]
This is impossible as $b\ne 0$. Then $|\mu|\in(0, 1/3]$ for the case $b\ne 0$ and $\mu>0$.

(ii) $\mu<0$.
Applying a similar proof as in Case 2(i) shows that
\[\mu E_\la+3^p (b+\mu)-\mu\subseteq E_\la\quad\textrm{for all }p\in\N.\]
If $b+\mu\ne 0$, then $3^p(b+\mu)-\mu$ tends to infinity as $p\to\infty$, leading to a contradiction. We will deny the case $b+\mu=0$ either.

Suppose $b+\mu=0$. It follows from $0=\mu\cdot 1-\mu \in \mu E_{\la}-\mu$ and Lemma \ref{lem:geometric structure-1} that $\mu E_{\la}-\mu\subseteq f_0(E_\la)\cup f_\la(E_\la)=\Lambda_L
\cup \Lambda_M\cup \Lambda_R$.
%
We claim that
\begin{equation}\label{LM}
f_0(E_\la)=\mu\cdot\frac {E_\la}3-\mu\subseteq\Lambda_M\quad{\textrm{and}}\quad f_\la(E_\la)=\mu\cdot\frac {E_\la+\la}3-\mu\subseteq\Lambda_M.
 \end{equation}

Observe that  the largest gap in $\mu\cdot\frac {E_\la}3-\mu$ and $\mu\cdot\frac {E_\la+\la}3-\mu$  are both of length $|\mu|\cdot\frac L3<\frac L3$. Furthermore,    the distances between $\Lambda_L$ and $\Lambda_M$, between $\Lambda_M$ and $\Lambda_R$ are both of length $\frac L3$. So   each one of $\mu\cdot\frac {E_\la}3-\mu$ and $\mu\cdot\frac {E_\la+\la}3-\mu$ must belong to   one of $\Lambda_L,\Lambda_M$ and $\Lambda_R$.

Suppose  on the contrary to (\ref{LM}) that $\mu\cdot\frac {E_\la}3-\mu$ is a subset of either $\Lambda_L$ or $\Lambda_R$. In the former case, we have $\mu E_{\la}-\mu\subseteq \Lambda_L$ as $\mu<0$. However, this is impossible as the length of $\mu E_{\la}-\mu$ is $|\mu|>\frac 13$ while the length of $\Lambda_L$ is $\frac{\lambda+1}9<\frac 13$. The latter case can not happen either as the length of $\mu\cdot\frac {E_\la}3-\mu$ is  $\frac{|\mu|}{3}>\frac 19$ while the length of $\Lambda_R$ is $\frac 19$. This proves the first inclusion in (\ref{LM}).

For the second inclusion,
it follows from $\mu\cdot\frac {E_\la}3-\mu\subseteq\Lambda_M$ and $\mu<0$ that  $\mu\cdot\frac {E_\la+\la}3-\mu$ is a subset of either $\Lambda_L$ or $\Lambda_M$.
Suppose the former case happens, then the length of the largest gap of $\mu\left(\frac {E_\la}3\cup \frac {E_\la+\la}3\right)-\mu$ should be at least the length of the distance between $\Lambda_L$ and $\Lambda_M$, that is, $\mu\cdot\frac L3\ge \frac L3$, leading to a contradiction.  Therefore, we establish (\ref{LM}).

Note that the left end point of $\Lambda_M$ is $\frac 29$ and the length of $\Lambda_M$ is $\frac{4\la-1}{9}$. By (\ref{LM}) we have $\left(\mu\cdot\frac {E_\la+\la}3-\mu\right)_{\min}=\mu\cdot\frac{1+\la}3-\mu\ge  \frac{2}{9} $. This together with $\la=1-3^{-m}\ge \frac 23$ yields $|\mu|=-\mu\ge \frac 12$. Again by (\ref{LM}), the length of $\frac {E_\la}3\cup \frac {E_\la+\la}3$ should be no more than that of $\Lambda_M$. Therefore we have $\frac{\la+1}6\le |\mu|\cdot\frac{\la+1}3\le |\Lambda_M|=\frac{4\lambda-1}{9} $. This gives $\la\ge 1$, leading to a contradiction.
\end{proof}

Suppose $g(E_\la)$ is a subset of $E_\la$, then by Lemma \ref{lem:geometric structure-1}  we have $0<|\mu|\le\frac 13$. There are two distinct cases depending  on whether $0<|\mu|<\frac 13$ or $|\mu|=\frac 13$, each requiring a separate argument.

\begin{lemma}\label{|mu|<1/3}
If $g(E_\la)=\mu E_\la+b\subseteq E_\la$ and $0<|\mu|<\frac 13$, then $3\mu E_\la+c\subseteq E_\la$ for some $c\in \R$.
\end{lemma}
\begin{proof} By Lemma \ref{lem:geometric structure-1} it follows that $g(E_\la)\subseteq f_0(E_\la)\cup f_\la(E_\la)$ or $g(E_\la)\subseteq f_2(E_\la)$. If $g(E_\la)\subseteq f_2(E_\la)$, i.e., $\mu E_\la+b\subseteq \frac{E_\la+2}{3}$, then $3\mu E_\la+3b-2\subseteq E_\la$, and we are done by taking $c=3b-2$. In the following we assume $g(E_\la)\subseteq f_0(E_\la)\cup f_\la(E_\la)$.

In view of Figure \ref{fig:1}, the hole between $\La_L$ and $\La_M$ has length $L/3$, where $L=(1-\la)/3$ is the length of the largest hole in $E_\la$. Note that the length of the largest hole in $g(E_\la)$ is $|\mu|L<L/3$. Then it follows that $g(E_\la)$ can not intersect both $\La_L$ and $\La_M$. Similarly, we have $g(E_\la)$ can not intersect both $\La_M$ and $\La_R$. Therefore, $g(E_\la)$ is contained in one of the following intervals: $\La_L, \La_M$ and $\La_R$. We split the proof into the following three cases.

(I) $g(E_\la)\subseteq \La_L$. Then $g(E_\la)\subseteq f_0(\De)\cap E_\la=f_0(E_\la)$, where the second equality follows from the totally self-similarity of $E_\la$. So,
\[
\mu E_\la+b\subseteq \frac{E_\la}{3},
\]
which gives $3\mu E_\la+3b \subseteq E_\la$. Hence, the lemma holds in this case with $c=3b$.

(II) $g(E_\la)\subseteq \La_R$. Then $g(E_\la)\subseteq f_\la(\De)\cap E_\la=f_\la(E_\la)$. So $\mu E_\la+b\subseteq (E_\la+\la)/3$. Again the lemma holds by taking $c=3b-\la$.

(III) $g(E_\la)\subseteq \La_M$. Then $b=g(0)\in \La_M=[f_{02}(0), f_{\la\la}(1)]=[2/9, (1+4\la)/9]$. We only prove the case $0<\mu<\frac 13$, the case
$-\frac 13<\mu<0$ can be proved in the same way.
 In terms of Lemma \ref{lem:negative-1} we distinguish two cases.
\begin{itemize}
\item $0\le b\le (1-\mu)/3$. Then by Lemma \ref{lem:negative-1} (i) we have $\mu E_\la+3b\subseteq E_\la$. Since $3b\ge 2/3$, by Lemma \ref{lem:gemometric structure-2} (iii) this implies
\[
\mu E_\la+3b\subseteq f_2(E_\la)=\frac{E_\la+2}{3}.
\]
Hence, $3\mu E_\la+9 b-2\subseteq E_\la$. This proves the lemma by taking $c=9b-2$.

\item $(1-\mu)/3<b\le (1+4\la)/9$. Then by Lemma \ref{lem:negative-1} (i) it follows that
\[
\mu E_\la+h_i(b)\subseteq E_\la\quad\textrm{for }i=1, 2,
\]
where
\begin{align*}
&h_1(x):=3x-d_1&\textrm{with}\quad& d_1:=\la-\la\mu,\\
&h_2(x):=3x-d_2&\textrm{with}\quad& d_2:=\la-2\mu.
\end{align*}

In view of Figure \ref{fig:1}, if $\mu E_\la+h_i(b)$ is not contained in $\La_M$, then the above argument gives that $3\mu E_\la+c\subseteq E_\la$ for some $c\in\R$.
So,
repeating the above argument, it follows that either  $3\mu E_\la+c\subseteq E_\la$ for some $c\in\R$, or
\begin{equation}\label{eq:416-1}
\mu E_\la+h_i^n(b)\subseteq E_\la\quad\textrm{for all }n\in\N\textrm{ and }i\in\set{1,2}.
\end{equation}
We will deny the latter case.

Observe that $h_i$ is an expanding map with fixed point $d_i/2$. Then  $h_i^n(b)=3^n(b-d_i/2)+d_i/2$ for all positive integer $n$. Since $d_1\ne d_2$, it follows that for any $b\in((1-\mu)/3, (1+4\la)/9]$ we have either $h_1^n(b)$ or $h_2^n(b)$ tends to infinity as $n\ra\f$. This leads to a contradiction with (\ref{eq:416-1}).
\end{itemize}

Therefore, there must exist $c\in\R$ such that $3\mu E_\la+c\subseteq E_\la$. Our proof is complete.
\end{proof}

The last case to consider is $|\mu|=\frac 13$. In this case, we will prove that the only possibility is $\mu=\frac 13$ and $b\in\set{f_0(0), f_\la(0), f_2(0)}$.
\begin{lemma}\label{|mu|=1/3}
If $g(E_\la)=\mu E_\la+b\subseteq E_\la$ and $|\mu|=\frac 13$, then $\mu=\frac 13$ and $b\in\{f_0(0),f_{\la}(0),
f_2(0)\}$.
\end{lemma}
\begin{proof}

First we prove that $g(E_\la)=-\frac{E_\la}{3}+b$ cannot be a subset of  $E_\la$ for any $b\in {\mathbb R}$.
Suppose on the contrary that $g(E_\la)=-\frac{E_\la}{3}+b\subseteq E_\la$. By Lemma \ref{lem:geometric structure-1} it follows that either $g(E_\la)\subseteq f_0(E_\la)\cup f_\la(E_\la)$ or $g(E_\la)\subseteq f_2(E_\la)$.
We claim that $g(E_\la)\nsubseteq f_2(E_\la)$. For otherwise we have $-E_\la+3b-2\subseteq E_\la$, which forces $-E_\la+3b-2=E_\la$. This contradicts the asymmetry  of $E_\la$.

Now we assume $g(E_\la)\subseteq f_0(E_\la)\cup f_\la(E_\la)$. In view of Figure \ref{fig:1}, one can check that the lengths of $\Lambda_L,\Lambda_M$
and $\Lambda_R$ are strictly less than $1/3$. So, either $g(E_\la)$   intersects both  $\La_L$ and $\La_M$, or $g(E_\la)$   intersects both $\La_M$ and $\La_R$. Since the gap $f_0(H)$ between $\La_L$ and $\La_M$ has length $L/3$ (recall that $L$ is the length of the largest hole in $E_\la$), and the gap $f_\la(H)$ between $\La_M$ and $\La_R$ also has length $L/3$, it follows that  the (unique) largest gap $g(H)$ in $g(E_\la)$ is either equal to $f_0(H)$, or equal to $f_\la(H)$. Therefore, we have $g(2/3)$ equals to either $f_0(\frac{\lambda+1}3)$ or $f_\lambda(\frac{\lambda+1}3)$, i.e., $b=1/3+\la/9$ or $b=(1+\la)/3+\la/9$.

Case (I). $b=1/3+\la/9$. Take $y=(2\la 0^{m-1}2^\f)_3\in E_\la$, where $m$ is the integer such that $\la=1-3^{-m}$. We claim $g(y)\in f_{0\la}(H)=((0\la\la2^\f)_3, (0\la 20^\f)_3)$. This can be verified by the following calculation:
\begin{align*}
g(y)=-(02\la 0^{m-1}\la 2^\f)_3+(0(2+\la)2^\f)_3&=(0\la(2-\la)2^{m-1}(2-\la)0^\f)_3\\
&<(0\la (2-\la)2^m0^\f)_3=(0\la 20^\f)_3,
\end{align*}
with the last equality using that $\la=(2^m0^\f)_3$. On the other hand,
\begin{align*}
g(y)=(0\la(2-\la)2^{m-1}(2-\la)0^\f)_3>(0\la\la2^\f)_3,
\end{align*}
where the inequality follows from $((2-\la)(2-\la)0^\f)_3>(\la2^\f)_3$ if $\la =1-3^{-1}$ and
\begin{align*}
(2^{m-1}(2-\la)0^\f)_3&>(2^{m-1}02^\f)_3=1-2/3^m=(\la+1)-(2-\la)\\
&>(\lambda+1)/3=(\la2^\f)_3
\end{align*}
if $\la =1-3^{-m}$ with $m\ge 2$.
This proves   $g(y)\notin E_\la$, leading  to a contradiction with our hypothesis $g(E_\la)\subseteq E_\la$.

Case (II). $b=(1+\la)/3+\la/9$. It follows from Case (I) that $g(y)-\la/3\in f_{0\la}(H)$. Thus $g(y)\in f_{\la\la}(H)$, which implies that $g(y)\notin E_\la$.

Therefore, we conclude that $g(E_\la)=-\frac{E_\la}{3}+b\nsubseteq E_\la$ for all $b\in\R$.

In the following we will determine for which kind of $b$ we have $\frac{E_\la}{3}+b\subseteq E_\la$.
Observe that $b=g(0)\in E_\la=f_0(E_\la)\cup f_\la(E_\la)\cup f_2(E_\la)$. We split the proof into the following two cases.

Case (I). $b\in f_2(E_\la)$. Then by Lemma \ref{lem:gemometric structure-2} (iii) we have $g(E_\la)\subseteq f_2(E_\la)$, which implies
\[
b=g(E_\la)_{\min}\ge \frac{2}{3}\quad\textrm{and}\quad \frac{1}{3}+b=g(E_\la)_{\max}\le 1.
\]
So, $b=2/3=f_2(0)$.

Case (II). $b\in f_0(E_\la)\cup f_\la(E_\la)$. Then by Lemma \ref{lem:geometric structure-1} it follows that $g(E_\la)=\frac{E_\la}{3}+b\subseteq f_0(E_\la)\cup f_\la(E_\la)$. We claim that  $b\in \{f_0(0),f_{\la}(0)\}
$

Write $\mu=1/3$. If $b\ge (1-\mu)/3$, then by Lemma \ref{lem:negative-1} (i) we have $\mu E_\la+3b+2\mu-\la\subseteq E_\la$. Since
\[
(\mu E_\la+3b+2\mu-\la)_{\max}=3\mu+3b-\la\ge 1+2\mu-\la>2\mu=\frac{2}{3},
\]
by Lemma \ref{lem:gemometric structure-2} it follows that $\mu E_\la+3b+2\mu-\la\subseteq f_2(E_\la)$. By Case (I) we conclude that $3b+2\mu-\la=2/3$. Using $\mu=1/3$ we obtain $b=\la/3=f_\la(0)$ as desired.

If $b<(1-\mu)/3$, then again by Lemma \ref{lem:negative-1} (i) we have $\mu E_\la+3b\subseteq E_\la$. Since $(\mu E_\la+3b)_{\min}=3b<1-\mu=2/3$, we conclude that $\mu E_\la+3b\subseteq f_0(E_\la)\cup f_\la(E_\la)$.

Repeating the above argument infinitely it follows that either
\begin{equation}\label{eq:trans-1}
b=\frac{\la}{3^p}\quad\textrm{for some }p\in\N,
\end{equation}
or
\begin{equation}\label{eq:trans-2}
\mu E_\la+3^p b\subseteq f_0(E_\la)\cup f_\la(E_\la)\quad\textrm{for all }p\in\N.
\end{equation}
Clearly, (\ref{eq:trans-2}) implies that $b=0=f_0(0)$. In view of (\ref{eq:trans-1}) it suffices to prove $b\ne \la/3^p$ for any $p\ge 2$.

Suppose on the contrary that $b=\la/3^p$ for some integer $p\ge 2$. Then
\[
g(f_\la(E_\la))_{\max}=\left(\frac{1}{3}\left(\frac{E_\la+\la}{3}\right)+\frac{\la}{3^p}\right)_{\max}=\frac{1+\la}{9}+\frac{\la}{3^p}\le \frac{1+2\la}{9}<\frac{1}{3}.
\]
By Lemma \ref{lem:gemometric structure-2} (ii) it follows that $g(f_\la(E_\la))\subseteq f_0(E_\la)$. This implies
\begin{equation}\label{eq:trans-3}
\Gamma:=\frac{E_\la+\la}{3}+\frac{\la}{3^{p-1}}\subseteq E_\la.
\end{equation}
In view of Figure \ref{fig:1} this is impossible, since $\Gamma_{\max}>{(1+\la)}/{3}$ and $\Gamma_{\min}<2/3$.
\end{proof}

Now we are ready to prove Proposition \ref{prop:positive}. 
\begin{proof}[Proof of Proposition \ref{prop:positive}]
 By Lemma \ref{lem:step1-1} and Lemma \ref{|mu|=1/3} we have either $0<|\mu|< 1/3$ or $\mu=1/3$. Suppose $\mu\ne 3^{-n}$ for any positive integer $n$. Note by Lemma \ref{|mu|=1/3} that $\mu\ne-1/3$. Then there exists a positive integer $k$ such that $3^{-k-1}<\mu<3^{-k}$ or $3^{-k-1}\le -\mu< 3^{-k}$. Using Lemma \ref{|mu|<1/3} for $k$ times yields
\[3^k\mu E_\la+c_k\subseteq E_\la\quad\textrm{for some }c_k\in\R.
\]
This implies that  either $|3^k\mu|\in(1/3, 1)$ or $-E_\la/3+c_k\subseteq E_\la$, which is impossible by Lemma \ref{lem:step1-1} and Lemma \ref{|mu|=1/3}. This completes the proof.
\end{proof}

To prove Theorem \ref{th:3}, we still need to determine the translation parameter $b$ for which $E_\la/3^n+b\subseteq E_\la$.

\begin{proof}[Proof of Theorem \ref{th:3}]
Let $g(E_\la)=\mu E_\la+b\subseteq E_\la$ with $0<|\mu|<1$ and $b\in\R$. By Proposition \ref{prop:positive} we have $\mu=3^{-n}$ for some $n\in\N$. So it suffices to prove $b\in\set{f_{\mathbf i}(0): \mathbf i\in\Om^n}$. We do this now by induction on $n$.

When $n=1$ this has been proven by Lemma \ref{|mu|=1/3}. Now let $n\ge 1$ and suppose $b\in\set{f_{\mathbf i}(0): \mathbf i\in\Om^n}$ for $\mu=3^{-n}$. We will prove that $b\in\set{f_{\mathbf i}(0): \mathbf i\in\Om^{n+1}}$ for $\mu=3^{-(n+1)}$.

Take $\mu=3^{-(n+1)}$. Observe that $b=g(0)\in E_\la=f_0(E_\la)\cup f_\la(E_\la)\cup f_2(E_\la)$. We distinguish two cases: (A) $b\in f_2(E_\la)$; (B) $b\in f_0(E_\la)\cup f_\la(E_\la)$.

Case (A). $b\in f_2(E_\la)$. Then by Lemma \ref{lem:gemometric structure-2} (iii) we have $g(E_\la)\subseteq f_2(E_\la)$, i.e.,
\[
\frac{E_\la}{3^{n+1}}+b\subseteq \frac{E_\la+2}{3}\quad\Longrightarrow\quad \frac{E_\la}{3^n}+3b-2\subseteq E_\la.
\]
 By the induction hypothesis there exists $\mathbf j\in\Om^n$ such that $3b-2=f_{\mathbf j}(0)$. So, $b=\frac{f_{\mathbf j}(0)+2}{3}=f_{2\mathbf j}(0)$ as required.

Case (B). $b\in f_0(E_\la)\cup f_\la(E_\la)$. Note that $b=g(E_\la)_{\min}$. If $b\ge \frac{\la}{3}$, then by Lemma \ref{lem:gemometric structure-2} (ii) we have $g(E_\la)\subseteq f_\la(E_\la)$, which implies $\frac{E_\la}{3^n}+3b-\la\subseteq E_\la$. By the induction hypothesis there exists $\mathbf j\in\Om^n$ such that $3b-\la=f_{\mathbf j}(0)$. Again, we obtain $b=\frac{f_{\mathbf j}(0)+\la}{3}=f_{\la\mathbf j}(0)$.

To finish the proof it suffices to consider the case $b<\la/3$. Recall that $\la=1-3^{-m}$ for some $m\in\N$ and $\mu=3^{-(n+1)}$. The case for $n\ge m$ is simpler than that for $n<m$. We consider them separately.

(B1). $n\ge m$. Then
\[
\left(\frac{E_\la}{3^{n+1}}+b\right)_{\max}<\frac{1}{3^{n+1}}+\frac{\la}{3}=\frac{1}{3}+\frac{1}{3^{n+1}}-\frac{1}{3^{m+1}}\le \frac{1}{3}.
\]
By Lemma \ref{lem:gemometric structure-2} (i) and the induction hypothesis it follows that $b=f_{0\mathbf j}(0)$ for some $\mathbf j\in\Om^n$.

(B2). $n<m$. Then $\la=1-3^{-m}\ge 1-3^{-(n+1)}=1-\mu=(2^{n+1}0^\f)_3$. Since $b=g(E_\la)_{\min}<\la/3=(02^m0^\f)_3$,   we   consider the following four cases.

\begin{itemize}
\item $b\le (02^n0^\f)_3$. Then
\[
\left(\frac{E_\la}{3^{n+1}}+b\right)_{\max}\le \frac{1}{3^{n+1}}+\frac1 3-\frac{1}{3^{n+1}}=\frac{1}{3}.
\]
So, by Lemma \ref{lem:gemometric structure-2} (i) and  the induction hypothesis it follows that there exists $\mathbf j\in\Om^n$ such that  $b=f_{0\mathbf j}(0)$.

\item $b=(02^n\la 0^\f)_3$. Then
\[
\left(\frac{f_\la(E_\la)}{3^{n+1}}+b\right)_{\max}=\frac{1+\la}{3^{n+2}}+\frac{1}{3}+\frac{\la}{3^{n+2}}-\frac{1}{3^{n+1}}<\frac{1}{3}.
\]
By Lemma \ref{lem:gemometric structure-2} (i) we have $\frac{f_\la(E_\la)}{3^{n+1}}+b\subseteq f_0(E_\la)$, which implies
\begin{equation}\label{eq:28-1}
F_n:=\frac{E_\la}{3^{n+1}}+\frac{2\la}{3^{n+1}}+1-\frac{1}{3^n}\subseteq E_\la.
\end{equation}
Observe that $(F_n)_{\min}=\frac{2\la}{3^{n+1}}+1-\frac{1}{3^n}>\frac{2}{3}$. By Lemma \ref{lem:gemometric structure-2} (iii) it follows that $F_n\subseteq f_2(E_\la)$, which implies
\begin{equation}\label{eq:28-2}
F_{n-1}:=\frac{E_\la}{3^{n}}+\frac{2\la}{3^n}+1-\frac{1}{3^{n-1}}\subseteq E_\la.
\end{equation}
 Repeating the same argument $n$ times, in view of (\ref{eq:28-1}) and (\ref{eq:28-2}), it follows that
 \[
 \frac{E_\la}{3}+\frac{2\la}{3}\subseteq E_\la.
 \]
 By Lemma \ref{|mu|=1/3} we deduce that $\frac{2\la}{3}\in\set{0, \frac{\la}{3}, \frac{2}{3}}$, leading to  a contradiction.

\item $b=(02^{n+1}0^\f)_3$. Then
\[
\left(\frac{f_2(E_\la)}{3^{n+1}}+b\right)_{\min}=\frac{2}{3^{n+2}}+\frac{1}{3}-\frac{1}{3^{n+2}}=\frac{1}{3}+\frac{1}{3^{n+2}}\in\left(\frac{1}{3}, \frac{2}{3}\right).
\]
By Lemma \ref{lem:gemometric structure-2} (ii) we have $\frac{f_2(E_\la)}{3^{n+1}}+b\subseteq f_\la(E_\la)$, which implies
\begin{equation}\label{eq:28-3}
\Gamma_n:=\frac{E_\la}{3^{n+1}}+\frac{1}{3^{n+1}}+\frac{1}{3^m}\subseteq E_\la.
\end{equation}
Since $1\le n<m$, we have $(\Gamma_n)_{\max}=\frac{2}{3^{n+1}}+\frac{1}{3^m}\le \frac{1}{3}$. By Lemma \ref{lem:gemometric structure-2} (i) it follows that $\Gamma_n\subseteq f_0(E_\la)$, i.e.,
\begin{equation}\label{eq:28-4}
\Gamma_{n-1}:=\frac{E_\la}{3^n}+\frac{1}{3^n}+\frac{1}{3^{m-1}}\subseteq E_\la.
\end{equation}
Iterating this argument $n$ times, in view of (\ref{eq:28-3}) and (\ref{eq:28-4}), we conclude that
\[
\frac{E_\la}{3}+\frac{1}{3}+\frac{1}{3^{m-n}}\subseteq E_\la.
\]
By Lemma \ref{|mu|=1/3} this can only happen when $m=n+1$, and in this case we have $b=f_{02^m}(0)=f_{\la0^{n}}(0)$ with $\la 0^n\in\Omega^{n+1}$.

\item $(02^{n+1}0^\f)_3<b<(02^m 0^\f)_3$. Then $m>n+1$ and
\[
b=(02^{n+1}0^\f)_3+\frac{\ep}{3^{n+2}}\quad\textrm{with}\quad 0<\ep<(2^{m-n-1}0^\f)_3.
\]
By the same argument as in the case $b=(02^{n+1}0^\f)_3$ we can prove that $\frac{f_2(E_\la)}{3^{n+1}}+b\subseteq f_\la(E_\la)$, which implies
\begin{equation}\label{eq:29-1}
G_n:=\frac{E_\la}{3^{n+1}}+\frac{1+\ep}{3^{n+1}}+\frac{1}{3^m}\subseteq E_\la.
\end{equation}
Since $\ep<(2^{m-n-1}0^\f)_3$, we have  $(G_n)_{\max}=\frac{2+\ep}{3^{n+1}}+\frac{1}{3^m}<3^{-n}\le 1/3$. By Lemma \ref{lem:gemometric structure-2} (i) it follows that $G_n\subseteq f_0(E_\la)$, which implies
\begin{equation}\label{eq:29-2}
G_{n-1}:=\frac{E_\la}{3^n}+\frac{1+\ep}{3^n}+\frac{1}{3^{m-1}}\subseteq E_\la.
\end{equation}
Repeating this argument we conclude that
\[
\frac{E_\la}{3}+\frac{1+\ep}{3}+\frac{1}{3^{m-n}}\subseteq E_\la.
\]
By Lemma \ref{|mu|=1/3} it follows that
\[\frac{1+\ep}{3}+\frac{1}{3^{m-n}}\in\set{0, \frac{\la}{3}, \frac{2}{3}},\]
which is impossible since $\frac{1+\ep}{3}+\frac{1}{3^{m-n}}\in(\frac{1}{3}, \frac{2}{3})$.
\end{itemize}

By Cases (A) and (B) we prove that for $\mu=3^{n+1}$ the translation parameter $b=f_{\bf i}(0)$ for some ${\bf i}\in\Omega^{n+1}$. By induction this completes the proof.
\end{proof}

\begin{remark}
Although  each affine map $g$ satisfying $g(E_\la)\subseteq E_\la$ must be of the form $f_{\bf i}$. This is not necessarily true for any other totally self-similar set. A counterexample can be found  in   \cite[Theorem 6.2]{Ele-Kel-Mat-10}.
\end{remark}

 \section{Unique expansion, multiple expansions and {final remarks}}\label{th:5}

In this section we will determine the size of the self-similar set $E_\la$ when it is totally self-similar, and the size of the set of points having finite triadic codings with respect to the alphabet $\Om=\set{0, \la, 2}$ as well. For $k\in\N\cup\set{\aleph_0, 2^{\aleph_0}}$ let
\[
\u_\la^{(k)}:=\set{x\in E_\la: x\textrm{ has precisely }k\textrm{ different triadic codings} }.
\]
Then for each $x\in \u_\la^{(k)}$ there exist precisely $k$ different sequences $(d_i)\in\Om^\N$ such that $x=((d_i))_3$. In particular, for $k=1$ the set $\u_\la^{(1)}$ contains all points with a unique triadic coding.

%
%
%
  \begin{proof}[Proof of Theorem \ref{th:4}]
 Take $\la=1-3^{-m}$. By Theorem \ref{th:1} it follows that $E_\la$ is totally self-similar. Then by Proposition \ref{prop:equivalent-condition-totally-self-similar} one can verify that
 \begin{equation}\label{eq:intersection}
f_{02^m}=f_{\la 0^m}\quad\textrm{and}\quad f_0(E_\la)\cap f_\la(E_\la)=f_{02^m}(E_\la).
 \end{equation}
 So $E_\la$ is a graph-directed set satisfying the open set condition (cf.~\cite{Ngai-Wang-01}). More precisely, let $X$ be the subshift of finite type over the alphabet $\set{0, \la, 2}$ with the forbidden block $02^m$. Then
 \[E_\la=\set{((d_i))_3: (d_i)\in X}.\]
  Let $s:=\dim_H E_\la$ and let $\mathcal H^s(\cdot)$ be the $s$-dimensional Hausdorff measure. Observe that $f_0(E_\la)\cap f_2(E_\la)=f_\la(E_\la)\cap f_2(E_\la)=\emptyset$. Therefore,  by (\ref{eq:intersection}) it follows that
 \begin{equation}\label{eq:dim-1}
 \begin{split}
 \mathcal H^s(E_\la)&=\sum_{d\in\Om}\mathcal H^s(f_d(E_\la))-\mathcal H^s(f_0(E_\la)\cap f_\la(E_\la))\\
 &=3^{1-s}\mathcal H^s(E_\la)-3^{-(m+1)s}\mathcal H^s(E_\la).
 \end{split}
 \end{equation}
 Note that $X$ is a transitive subshift of finite type, i.e., for any two admissible words in $X$ we can find a sequence in $X$ containing both of them. Then $\mathcal H^s(E_\la)\in(0, \f)$ (cf.~\cite{Mauldin-Williams-88}).
 Hence, (\ref{eq:dim-1}) implies that the Hausdorff dimension $\dim_H E_\la=s$  satisfies
 \[
 1=3^{1-s}-3^{-(m+1)s}.
 \]

 For the dimension of $\u_\la:=\u_\la^{(1)}$ we observe by (\ref{eq:intersection}) that any $x\in f_0(E_\la)\cap f_\la(E_\la)$ has at least two triadic codings. So, by the same argument as above $\u_\la$ is also a graph-directed set satisfying the open set condition. But in this case the underline  subshift of finite type has forbidden blocks $02^m$ and $\la 0^m$. Let $t:=\dim_H\u_\la$. Then  by (\ref{eq:intersection}) the $t$-dimensional Hausdorff measure of  $\u_\la$ satisfies
 \begin{equation}\label{eq:dim-2}
 \begin{split}
 \mathcal H^t(\u_\la)&=\sum_{d\in\Om}\mathcal H^t(f_d(\u_\la))-2\mathcal H^t(f_0(\u_\la)\cap f_\la(\u_\la))\\
 &=3^{1-t}\mathcal H^t(\u_\la)-2\mathcal H^t(f_{02^m}(\u_\la))\\
 &=3^{1-t}\mathcal H^t(\u_\la)-2\cdot 3^{-(m+1)t}\mathcal H^t(\u_\la).
 \end{split}
 \end{equation}
 Note that $\mathcal H^t(\u_\la)\in(0, 1)$. This implies that   $\dim_H\u_\la=t$  is given by
 \[
 1=3^{1-t}-2\cdot 3^{-(m+1)t}.
 \]
Clearly, $\dim_H\u_\la<\dim_H E_\la$.

 For any  integer $k\ge 2$, by Theorems 1.1 and 1.5 in \cite{KKKL-2} it follows that  the Hausdorff dimension of   $\u_\la^{(k)}$ is the same as that of $\u_\la$, i.e., $\dim_H\u_\la^{(k)}=\dim_H\u_\la=t.$ Moreover, $\u_\la^{(\aleph_0)}$ is countably infinite, and then the dimension of $\u_\la^{(2^{\aleph_0})}$ is equal to $\dim_H E_\la$. This completes the proof.
 \end{proof}

 At the end of this section we pose some  questions.

\begin{question} Can we generalize the model studied in this paper to a two-parameter family
 \[
 f_0(x)=\rho x,\quad f_1(x)=\rho (x+\la)\quad\textrm{and}\quad f_2(x)=\rho x+1-\rho,
 \]
 where $0<\rho,\la<1$ and $\rho(2+\la)<1$? Clearly, the convex hull of the attractor  is the unit interval $I:=[0,1]$. Moreover, $f_0(I)\cap f_1(I)\ne \emptyset$ and $f_1(I)\cap f_2(I)=\emptyset$. What can we say about its spectrum, all its generating IFSs etc.?
 \end{question}

\begin{question} Let $E$ be a totally self-similar set generated by the IFS $\{\phi_i\}_{i=1}^N$. Suppose $\phi_i(E)\cap \phi_j(E)\ne \emptyset$ for some $i,j\in\{1,\cdots,N\}$. Does this imply that there exist two different finite words ${\bf i,j}\in\Omega^*$ such that  $\phi_{\bf i}=\phi_{\bf j}$?
\end{question}

Note by Theorem \ref{th:2} that the spectrum $l_\la$ vanishes if and only if $\la$ is irrational. Furthermore, we know from Theorem \ref{th:2} (ii) that $l_\la$ is computable if $\la$ is rational. However, only in a few cases we can explicitly determine the value $l_\la$.

\begin{question} Can we describe the spectrum set
\[\Lambda:=\set{l_\la: \la\in(0, 1)}?\]
Is $\Lambda$ a discrete set? Is it  closed?
\end{question}

\section*{Acknowledgement} The authors wish to thank Professor Y. Wang and Professor D.-J. Feng for some useful discussions. The second author was supported by NSFC No.~11401516. He would like to thank the Mathematical Institute of Leiden University. The third author was supported by the China Scholarship Council with grant number
201706745005.





\end{document}